\documentclass[reqno, 11pt]{amsart}
\usepackage{amsmath,amssymb,amsfonts,amsthm,graphicx}
\usepackage{bbm}
\usepackage{mathrsfs}
\usepackage{hyperref}
\usepackage{mathtools}
\usepackage[english]{babel}
\usepackage{tikz, graphicx}
\usepackage{microtype}
\usepackage[mathscr]{euscript}
\usepackage{calrsfs}
\usepackage[normalem]{ulem}
\usetikzlibrary{shapes.geometric, arrows.meta, positioning, patterns, 
                shadows, decorations.pathmorphing, fadings, calc}
\usepackage[margin=2.0cm]{geometry}
\usepackage{helvet}
\usepackage{float}

\usepackage{ulem}
\usepackage{mdframed}
\usepackage{tcolorbox}
\usepackage{lipsum}
\usepackage{enumitem}
\setlist{topsep=2pt, itemsep=0pt, leftmargin=10pt}

\usepackage[mathscr]{euscript}
\usepackage{calrsfs}
\usepackage{type1cm,eso-pic}
\def\DD{L^2(D)}

\newcommand{\thetaa}{\vartheta}

\newcommand{\nn}{\nonumber}

\newcommand{\ba}{\begin{array}}
\newcommand{\ea}{\end{array}}

\def\bge{\begin{eqnarray}}
\def\bgee{\begin{eqnarray*}}
\def\ege{\end{eqnarray}}
\def\egee{\end{eqnarray*}}

\def\ve{\varepsilon}

\def\u{z}

\author[N.I. Kavallaris]{Nikos I. Kavallaris}
\address{
 Department of Mathematics and Computer Science, Karlstad University, Karlstad, Sweden}
\email{nikos.kavallaris@kau.se}
\thanks{Corresponding author: N.I. Kavallaris (\texttt{nikos.kavallaris@kau.se})}

\author[C.V. Nikolopoulos]{Christos V. Nikolopoulos}
\address{Department of Mathematics, University of Aegean, Karlovassi, Samos, Greece}
\email{cnikolo@aegean.gr}

\author[S. Sankar]{Subramani Sankar}
\address{Department of Mathematics, Periyar University, Salem, India}
\email{subusankar27@gmail.com}
\allowdisplaybreaks

\usepackage{xcolor}
\allowdisplaybreaks
\usepackage[pagewise]{lineno}
\usepackage{graphicx,eurosym}
\usepackage{hyperref}
\usepackage{mathtools}
\colorlet{darkblue}{blue!50!black}
\hypersetup{
colorlinks,%
citecolor=red,%
filecolor=red,%
linkcolor=blue,%
urlcolor=blue,%
pdfnewwindow=true,%
pdfstartview={FitH}}
\DeclareMathOperator*{\esssup}{ess\,sup}

\let\originalleft\left
\let\originalright\right
\renewcommand{\left}{\mathopen{}\mathclose\bgroup\originalleft}
\renewcommand{\right}{\aftergroup\egroup\originalright}


\begin{document}
\title[Quenching estimates for a coupled stochastic PDE system with mixed noises]{
Quenching time and probability estimates for a stochastic reaction-diffusion system with coupled inner singular absorption terms  driven by mixed noises}
\date{\today}
	\maketitle \setcounter{page}{1} \numberwithin{equation}{section}
	\newtheorem{theorem}{Theorem}[section]
	\newtheorem{assumption}{Assumption}
	\newtheorem{lemma}{Lemma}[section]
	\newtheorem{Pro}{Proposition}[section]
	\newtheorem{Ass}{Assumption}[section]
	\newtheorem{Def}{Definition}[section]
	\newtheorem{Rem}{Remark}[section]
	\newtheorem{cor}{Corollary}[section]
	\newtheorem{proposition}{Proposition}[section] 
	\newtheorem{app}{Appendix:}
	\newtheorem{ack}{Acknowledgement:}

\begin{abstract}
This paper investigates a stochastic parabolic system under Robin boundary conditions, for which the deterministic counterpart exhibits finite quenching. The stochastic system incorporates mixed noise, combining standard one-dimensional Brownian motion and fractional Brownian motion. Under appropriate assumptions, we derive explicit lower and upper bounds for the quenching time of the solution and establish the global existence of a weak solution. Leveraging Malliavin calculus, we further obtain  a quantifiable lower and upper bound on the quenching probability. To complement the theoretical analysis, we design a numerical scheme tailored to the system and present results that validate the analytical predictions, offering insights into the interplay between noise and quenching behaviour.

		
\end{abstract}

	\tableofcontents
	\section{Introduction and state of the art}
In the current work we investigate the following stochastic coupled reaction-diffusion system:
 \begin{eqnarray}
&&du_{1}(t,x)=\left[\Delta u_{1}(t,x)+\frac{\lambda_{11}}{(1-u_{1}(t,x))^{2}}+\frac{\lambda_{12}}{(1-u_{2}(t,x))^{2}}\right]dt+(1-u_1(t,x)) dN_1(t),x\in D, \; t>0,\qquad\label{b1}\\
	&&du_{2}(t,x)=\left[\Delta u_{2}(t,x)+\frac{\lambda_{21}}{(1-u_{2}(t,x))^{2}} +\frac{\lambda_{22}}{(1-u_{1}(t,x))^{2}}\right]dt+ (1-u_{2}(t,x))dN_2(t),x\in D, \; t>0,\qquad\label{b1a}\\
   &&\frac{ \partial u_{i}(t,x)}{\partial \nu}+\beta u_{i}(t,x)=\beta, \ \ t>0, \ \ x\in \partial D, \ \ i=1,2,\label{b1c} \\
 && 0 \leq u_{i}(0,x)=f_{i}(x)<1, \ \ x\in D, \ \ i=1,2,\label{b1d}
	\end{eqnarray}
where  \( D \subset \mathbb{R}^d \) with \( d \geq 1 \), and smooth boundary \( \partial D \). The constants \( \beta, \lambda_{i1}, \lambda_{i2} \) for \( i = 1,2 \) are assumed to be positive real numbers. For each \( x \in \partial D \), \( \nu = \nu(x) \) denotes the outward unit normal vector to the boundary \( \partial D \).

The noise terms \( \{N_i(t) : t \geq 0\} \) appearing in the system are defined as
\begin{equation}\label{cnt}
N_i(t) := \int_0^t k_{i1} \, dW(s) + \int_0^t k_{i2} \, dB^H(s), \quad i = 1,2,
\end{equation}
for some positive constants \( k_{i1}, k_{i2}, i=1,2 \). These processes are mixtures of a standard one-dimensional Wiener process \( \{W(t) :t \geq 0\} \) and an one-dimensional fractional Brownian motion (fBm) \( \{B^H(t): t \geq 0\} \), both defined on a filtered probability space \( \left( \Omega, \mathcal{F}, (\mathcal{F}_t)_{t \geq 0}, \mathbb{P} \right) \); see also~\eqref{cnt}.

The fractional Brownian motion \( B^H(t) \) is characterized by a Hurst index \( H \in \left(\frac{1}{2}, 1\right) \), which governs the correlation of the increments and the pathwise regularity of the process (see, e.g.,~\cite{mis2008}). Values of \( H > \frac{1}{2} \) correspond to positively correlated increments, resulting in smoother sample paths.

Furthermore, the initial conditions \( f_1 \) and \( f_2 \) are assumed to be non-negative functions in \( C^2(D) \), not identically zero.

System~\eqref{b1}--\eqref{b1d} is motivated by the following prototypical model for electrostatically actuated micro-electro-mechanical systems (MEMS), which has been extensively studied in the literature (see, e.g., \cite{espos2010, kava2008, nikos2018}):
\begin{equation}\label{mems}
\left\{
\begin{aligned}
&\frac{\partial u}{\partial t} = \Delta u + \frac{\lambda}{(1 - u)^2}, && x \in D,\ t > 0, \\
&\frac{\partial u}{\partial \nu} + \beta u = \beta_c, && x \in \partial D,\ t > 0, \\
&0 \leq u(x,0) = u_0(x) < 1, && x \in D,
\end{aligned}
\right.
\end{equation}
Here, $u = u(t,x)$ denotes the deformation of an elastic membrane that constitutes part of the MEMS device. The parameters $\lambda$, $\beta$, and $\beta_c$ are positive real constants representing, respectively, the strength of the applied electrostatic voltage, the elastic restoring force at the boundary, and an external forcing at the boundary of the elastic membrane (see \cite{DKN21, DKN22} for further details). 

For additional mathematical models describing the operation of MEMS devices, including both local and nonlocal effects, we refer the interested reader to \cite{GRL22, GRL24, GRL24b, GK12, gkwy20, KLN16, KLNT11, KLNT15, M1, M2, M3, JAP-DHB02, PT01} and the references therein.

Modeling the interaction between multiple elastic membranes would naturally lead to a system of partial differential equations more complex than~\eqref{b1}--\eqref{b1d}, typically involving source terms of the form $F(u_1 - u_2)$ for some interaction function $F$. Nevertheless, the simplified system~\eqref{b1}--\eqref{b1d} can be viewed as a toy model that offers valuable insights into the quenching behaviour and the role of stochastic perturbations. In this context, the multiplicative (possibly fractional) noise terms of the form $(1 - u_i(t,x)) \, dN_i(t)$, for $i = 1,2$, are intended to capture the effect of correlated fluctuations in the physical parameters of the MEMS device.


In the limit \( k_{ij} \to 0^+ \) for \( i,j = 1,2 \), the stochastic system~\eqref{b1}--\eqref{b1d} reduces to its deterministic counterpart:
\begin{align*}
&\frac{\partial u_{1}}{\partial t}(t,x) = \Delta u_{1}(t,x) + \frac{\lambda_{11}}{(1 - u_{1}(t,x))^2} + \frac{\lambda_{12}}{(1 - u_{2}(t,x))^2}, \quad x \in D,\; t > 0, \\
&\frac{\partial u_{2}}{\partial t}(t,x) = \Delta u_{2}(t,x) + \frac{\lambda_{21}}{(1 - u_{2}(t,x))^2} + \frac{\lambda_{22}}{(1 - u_{1}(t,x))^2}, \quad x \in D,\; t > 0, \\
&\frac{\partial u_{i}(t,x)}{\partial \nu} + \beta u_{i}(t,x) = \beta, \quad x \in \partial D,\; t > 0,\; i = 1,2, \\
&0 \leq u_{i}(0,x) = f_{i}(x) < 1, \quad x \in D,\; i = 1,2.
\end{align*}
This deterministic version of system~\eqref{b1}--\eqref{b1d} has been extensively studied in the literature, particularly concerning its long-term dynamics, including the phenomenon of finite-time quenching (either simultaneous or non-simultaneous) as well as under various types of boundary conditions; see, for example,~\cite{JYW19, WZ22, ZW08, ZM10} and the references therein.

\subsection{State of the art}
To the best of our knowledge, the stochastic system~\eqref{b1}--\eqref{b1d} is introduced here for the first time. Consequently, the qualitative properties of its solutions, particularly the impact of stochastic perturbations on the system's evolution and quenching behaviour, remain uncharted in the existing literature. This work aims to fill this gap by systematically investigating how the presence of mixed noise, consisting of both Brownian and fractional Brownian components, modifies the dynamics of the system, with special focus on the mechanisms and characteristics of finite-time quenching.

Our approach is informed by a well-established body of research on quenching phenomena in both local and non-local single-equation models; see, for example,~\cite{DKN22, DNMK23,  K18, KNY24}. We also draw on insights from the extensive literature on blow-up behaviour in related stochastic models, including both single equations and coupled PDE systems with local or non-local structures; see, for instance,~\cite{doz2010, doz2013, doz2023, K15, KY20, li, skm1, smk2}. These foundational works guide the development of our analytical techniques and motivate our exploration of the novel features introduced by the stochastic terms.

The main objectives of this work are as follows:
\begin{enumerate}
    \item To establish finite-time quenching for the solution \( u = (u_1, u_2)^{\top} \) of system~\eqref{b1}--\eqref{b1d}.
    
    \item To provide sufficient conditions under which the solution \( u = (u_1, u_2)^{\top} \) to system~\eqref{b1}--\eqref{b1d} exists globally in time.
    
    \item To derive both lower and upper bounds for the quenching time and the probability of quenching for the solution \( u = (u_1, u_2)^{\top} \) to system~\eqref{b1}--\eqref{b1d}.
\end{enumerate}

To achieve these goals, we analyze the solution \( z = (z_1, z_2)^{\top} \) of the equivalent system~\eqref{zeq1}--\eqref{zeq4}.

\subsection{Layout of the paper}

The remaining sections are organized as follows:
    \begin{itemize}

\item The following section introduces the main mathematical concepts, key formulas, and foundational results from stochastic calculus that are employed throughout the manuscript.

    \item Section \ref{le} focuses on deriving local solutions for the original system \eqref{zeq1}--\eqref{zeq4} as well for a corresponding system of random partial differential equations (PDEs), see \eqref{rand1a}--\eqref{rand1c},  derived by applying a random transformation  to  \eqref{zeq1}--\eqref{zeq4}.  The transformed system facilitates the analysis required to establish bounds on the quenching time $\tau_q$ and the quenching probability.
        
    \item In Subsection~\ref{sec3}, Theorem~\ref{psk1a} establishes a rigorous lower bound \( \tau_* \) for the quenching time \( \tau_q \) associated with solutions to system~\eqref{zeq1}--\eqref{zeq4}. Furthermore, Corollary~\ref{sk3} identifies general conditions that guarantee global existence, while Proposition~\ref{ps1} analyzes more specific scenarios under which the solution to system~\eqref{zeq1}--\eqref{zeq4} exists for all times. In Subsection~\ref{sec4}, an upper bound for the quenching time is derived, both in a general setting (see Theorem~\ref{thm4.1}) and in a more specific context (see Corollary~\ref{ps2}). Based on these results, we also provide upper and lower estimates for the quenching rate in Theorem~\ref{Athm6.6}.
Within this framework, we utilize exponential functionals of the form
\[
\int_{0}^{t} \exp\left\{\rho_1 W(r) + \rho_2 B^{H}(r) + \sigma r \right\} \, dr,
\]
where $\rho_{11} = \rho_{12} := \rho_1$, $\rho_{21} = \rho_{22} := \rho_2$, and $\sigma$ is a suitably chosen positive constant.


 \item In Section~\ref{s5}, we obtain  upper and lower bounds for the (finite-time)  quenching probability of  solutions  of  system~\eqref{zeq1}--\eqref{zeq4}. To this end an upper bound for  quenching probability before a given fixed time $T>0$  is established first in Theorem \ref{thm6}. Further, in Theorem \ref{thm6t} the upper bounds for the tail of $\tau^{\ast}$ with the general dependent structure of the Brownian motion $\{W(t): t \geq 0\}$ and fBm $\{B^{H}(t): t \geq 0\}$ is given. Next, we explicitly provide a lower bound for the probability of finite-time quenching of solutions of~\eqref{zeq1}--\eqref{zeq4}, see  Theorem \ref{thm6.5}, for an appropriate choice of parameters, by using the Malliavin calculus and the method adopted in \cite{doz2023, KNY24}.  At the end of this section, we establish a result ensuring almost sure quenching, see Theorem~\ref{thm6.6}, in the special case where \( \frac{3}{4} < H < 1 \), whence fractional Brownian motion \( \{B^{H}(t) : t \geq 0\} \) is equivalent in law to the standard Brownian motion \( \{W(t) : t \geq 0\} \).

\item Finally, Section~\ref{ns} introduces a finite element scheme for the numerical approximation of the system~\eqref{b1}--\eqref{b1d}. The second part of the section presents a series of simulations that both validate the analytical results and clearly illustrate the influence of the mixed noise on the system's dynamics.

    \end{itemize}

 \section{Preliminaries}\label{prel}
In this section, we introduce the key mathematical concepts, formulas and function spaces that will be employed throughout the manuscript.

\subsubsection*{Fractional Brownian motion}To start with, a fractional Brownian motion of Hurst parameter $H \in (0,1)$ is a centered Gaussian process $\{B^{H}(t): t\geq 0\}$ with the covariance function (see \cite[Definition of 5.1, p.273]{nualart}) 
\begin{align*}
	R_{H}(t,s):=\mathbb{E}\left[ B^{H}(t)B^{H}(s)\right] =\frac{1}{2} \left( s^{2H}+t^{2H}-|t-s|^{2H} \right),
\end{align*}
so that $\mathbb{E}\left[|B^H(t)|^2\right]=t^{2H}$.	It is known that $\{B^{H}(t):t\geq 0\}$  admits the so called Volterra representation (for more details, see \cite{nualart}),
\begin{equation}\label{nk71}
	B^{H}(t):=\int_{0}^{t} K^{H}(t,s) dW(s),
\end{equation}  
where $\{W(t): t\geq 0\},$ is a standard Brownian motion and the Volterra kernel $K^{H}(t,s)$ is defined by 
\begin{align}\label{gsf58}
	K^{H}(t,s) = C_{H} \left[ \frac{t^{H-\frac{1}{2}}}{s^{H-\frac{1}{2}}}(t-s)^{H-\frac{1}{2}} -\left(H-\frac{1}{2}\right)\int_{s}^{t} \frac{u^{H-\frac{3}{2}}}{s^{H-\frac{1}{2}}}(u-s)^{H-\frac{1}{2}} du  \right], \ \ s \leq t,  
\end{align}
and $C_{H}$ is a constant depending only on the Hurst index  $H.$ In that case $\{W(t): t\geq 0\}$ and $\{B^{H}(t): t\geq 0\}$ are dependent processes.	We then express the auto-covariance function of the fBM in terms of
\begin{eqnarray}\label{aek21}
    R^{H}(t,s)=\int_{0}^{\min(s,t)} K^{H}(t,r)K_{H}(s,r) dr.
\end{eqnarray}
By It\^o isometry, we also have 
\begin{align*}
	\mathbb{E}\left[|B^H(t)|^2\right]=\int_0^t(K^H(t,s))^2ds.
\end{align*}
The fractional Brownian motion $\{B^{H}(t) :  t \ge 0\}$ is not a semimartingale for $H \ne 1/2$. It is interesting to note though, that the process $\{W(t) + B^{H}(t): t \ge 0\}$, when $H \in (3/4,1)$, is equivalent in law to Brownian motion, cf. \cite{cher2001}.

Throughout this work, we will use the Banach space $\mathcal{B}^{\gamma,2}\left(\left[0,t\right], L^2(D) \right)$, which consists of all measurable functions $u:[0,t] \rightarrow L^2(D)$ for which the norm $\Vert \cdot \Vert_{\gamma,2}$ is defined, i.e., 
  $$\Vert u \Vert_{\gamma,2} ^2= \left( \esssup_{s \in [0,t]} \Vert u(\cdot,s)\Vert_2 \right) ^2 + \int_0^t \left( \int_0^s \frac{\Vert u(\cdot,s)- u(\cdot,r)\Vert_2}{(s-r)^{\gamma+1}} dr \right)^2 ds <+ \infty, $$
 where $\Vert \cdot \Vert_2$ is the usual norm in $L^2(D)$,  cf. \cite{MN03, mis2008, Z}.  The requirement that $u \in \mathcal{B}^{\gamma, 2} \bigl( [0,t],L^2(D) \bigr)$ for some  $\gamma\in(1-H,1/2)$ 
 ensures that the stochastic integral $\int_0^t u(s) dB^H(s)$ exists as a generalized Stieltjes integral in the Young sense (see \cite{Z} and \cite[Proposition 1]{NV06}).

As far as the mixed processes $\{N_i(t):t \ge 0\},\; i=1,2$ are concerned, the stochastic integral with respect to $\{B^{H}(t):t \ge 0\}$ will be considered in the above pathwise sense, whereas the stochastic integral with respect to $\{W(t):t \ge 0\}$ will be considered in the It\^o sense, see also \cite{KNY24}. 

\subsubsection*{Malliavin derivative}Let $\mathcal{S}$ denote the space of step functions on the interval $[0,T]$, and define $\mathcal{H}$ as the closure of $\mathcal{S}$ with respect to the inner product
$
\langle \mathbf{1}_{[0,s]}, \mathbf{1}_{[0,t]} \rangle_{\mathcal{H}} := R^H(s,t),
$
recalling that $R^H(s,t)$ denotes the auto-covariance function given by \eqref{aek21}.

We begin by defining the map $\mathbf{1}_{[0,t]} \mapsto B_t^H=B^H(t)$ on $\mathcal{S}$ and then extend it to an isometry from $\mathcal{H}$ into the Gaussian space $\mathcal{H}_1(B_t^H)$ generated by the fractional Brownian motion. This isometry is denoted by $\phi \mapsto B_t^H(\phi)$.
 We now proceed to define the Malliavin derivative $D$ with respect to fBM. Let us consider a smooth cylindrical functional of the form $F = f(B_t^{H}(\phi))$, where $\phi \in \mathcal{H}$ and $f \in C_{b}^{\infty}(\mathbb{R})$. The Malliavin derivative $\mathcal{D}F$ is an $\mathcal{H}$-valued random variable defined via the duality relation
\begin{eqnarray*}
    \langle \mathcal{D}F, h \rangle_{\mathcal{H}} := f'(B_t^{H}(\phi)) \langle \phi, h \rangle_{\mathcal{H}} 
    = \left. \frac{d}{d\epsilon} f\big(B_t^{H}(\phi) + \epsilon \langle \phi, h \rangle_{\mathcal{H}} \big) \right|_{\epsilon = 0},
\end{eqnarray*}
which can be understood as a generalization of the directional derivative along the paths of fBm.

The operator $\mathcal{D}$ is closable as a mapping from $L^{p}(\Omega)$ into $L^{p}(\Omega; \mathcal{H})$ for any $p \geq 1$, and it allows for the construction of Sobolev-type spaces on the Wiener space. Among these, the space ${\mathbb{D}}^{1,2}$ is of particular importance. It is defined as the closure of the set of smooth cylindrical random variables with respect to the norm
\begin{eqnarray*}
    \| F \|_{{\mathbb{D}}^{1,2}} = \left( \mathbb{E}[|F|^2] + \mathbb{E}[\| \mathcal{D}F \|_{\mathcal{H}}^2] \right)^{1/2}.
\end{eqnarray*}
In our framework, the Malliavin derivative will be interpreted as a stochastic process $\{\mathcal{D}_t F : \, t \in [0,T] \}$ (see \cite{nualart}, Section 1.2.1 and Chapter 5).

\subsubsection*{Integration by parts, Itô formula and miscellaneous}
 Next we  review the integration by parts formula relevant to stochastic processes. Specifically, if $X(t)$ and $Y(t),\ t \in [0,T],\ T>0$ are It\^o stochastic processes defined by
	\begin{align*}
	X(t)&=X(0)+ \int_{0}^{t} \Psi(s) ds+ \int_{0}^{t} \Phi(s) dB^H(s),\ \mbox{and} \\
	Y(t)&=Y(0)+ \int_{0}^{t} \hat{\Psi}(s) ds+ \int_{0}^{t} \hat{\Phi}(s) dB^H(s),
	\end{align*}
	then the following integration by parts formula holds:
	\begin{align}\label{int2}
	X(t)Y(t)=X(0)Y(0)+ \int_{0}^{t} X(s)dY(s)+ \int_{0}^{t} Y(s)dX(s)+[X(t),Y(t)],\ \ t \in [0,T], 
	\end{align}
	where the last term in \eqref{int2} is the quadratic variation of $X(t)$, $Y(t)$ and is defined as
	\begin{align*}
	[X(t), Y(t)]:= \int_{0}^{t} \Phi(s)\hat{\Phi}(s)ds,
	\end{align*}
	for more details see \cite[Page 114]{KLEB2005}.
    
   Note that if $H>\frac{1}{2}$ in this case there is no quadratic covariation between the processes.  One can refer to \cite{biagini} for more about integration by parts formula for fBm. In this study, we are examining both mixed Brownian and fBm by using the It\^o's formula as follows:
   \begin{theorem}[{\cite[Theorem 2.7.2]{mis2008}}] \label{ito2}
	Let the process $X(t)=\displaystyle\sum_{i=1}^{m}\sigma_{i}B^{H_{i}}(t),$ where $H_{1}=1/2$ and $H_{i} \in (1/2,1)$ for $2 \leq i \leq m,\ \sigma_i$ are real numbers and the function $F \in C^{2}(\mathbb{R}).$ Then for any $t>0$, we have
	\begin{align}	F(X(t))&=F(X(0))+\sigma_{i}\int_{0}^{t}F'(X(s))dW(s)+\sum_{i=1}^{m}\sigma_{i}\int_{0}^{t}f'(X(s))dB^{H_{i}}(s)\nonumber\\
		&\qquad+\frac{\sigma_{i}^{2}}{2}\int_{0}^{t}F''(X(s))ds. \label{ifo}	 
	\end{align}
\end{theorem}
Hence, only the second derivatives will contribute to the It\^o formula for standard Brownian and fBm.
\noindent \\
Let us take  $$\bar{\mathcal{Z}}(t):=\displaystyle\int_{0}^{t}f(s)dW(s),$$ where $f$ is any continuous function. By applying It\^o's formula \eqref{ifo} to the process $\left\{ \exp\{\bar{\mathcal{Z}}(t)\}\right\}_{t \geq 0}$, we have
	\begin{align}
	\exp\{\bar{\mathcal{Z}}(t)\}=1+\int_{0}^{t} \exp\{\bar{\mathcal{Z}}(s)\} d\bar{\mathcal{Z}}(s)+\frac{1}{2}\int_{0}^{t} \exp\{\bar{\mathcal{Z}}(s)\}f^{2}(s)ds. \nonumber
	\end{align}
	Taking expectation on both sides, we get
	\begin{align}
	\mathbb{E}(\exp\{\bar{\mathcal{Z}}(t)\})=1+\frac{1}{2}\int_{0}^{t} \mathbb{E}\left( \exp\{\bar{\mathcal{Z}}(s)\}\right)f^{2}(s)ds. \nonumber 
	\end{align} 
	Therefore by taking $\bar{\mathcal{Y}}(t):=\mathbb{E}(\exp\{\bar{\mathcal{Z}}(t)\})$, and variation of constants formula  yields
	\begin{align}\label{f2}
	\mathbb{E}\left(\exp\left\lbrace  \int_{0}^{t} f(s)dW(s) \right\rbrace  \right)=\exp \left\lbrace \frac{1}{2}\int_{0}^{t} f^{2}(s)ds\right\rbrace.  
	\end{align}
\subsubsection*{Robin eigenvalue problem}
Let $\chi$ be the first eigenvalue of $-\Delta$ on $D$, which satisfies 
    \bge
	&&-\Delta \psi(x)=\chi \psi(x), \ \ x\in D,\label{a3} \\
    &&\frac{\partial \psi}{\partial \nu}(x)+\beta\psi(x)=0, \; x\in \partial D,\label{a3a}
\ege
	with $\psi$ being the corresponding eigenfunction, normalized so that 
    $\displaystyle\int_{D} \psi(x)dx=1.$ Then $\chi>0$  and $\psi$  is strictly positive on $D$ for $\beta>0,$ cf. \cite[Theorem 4.3]{Am} , and   by virtue of Jentsch's Theorem (see \cite[Theorem V.6.6]{Sc74}), we obtain 
    \begin{equation}\label{SK1}
    S_{t}\psi=\exp\{{-\chi t}\}\psi, t\geq 0,
    \end{equation} 
    where $\left\{ S_{t} \right\}_{t\geq 0}$  stands for the semigroup generated
by the operator $\mathcal{A}=-\Delta_R,$ i.e., the Laplace operator associated with homogeneous
Robin conditions (see e.g.,\cite[Chap. IV]{pazy}) and domain $D(\mathcal{A})= W ^{2,2}(D)\cap W^{1,2}(D).$   

	\section{Local existence}\label{le}
If we set  $z_{i}(t,x)=1-u_{i}(t,x), x \in D, t \geq 0, i=1,2$  then problem \eqref{b1}--\eqref{b1d} transforms into one with homogeneous boundary conditions 
\begin{eqnarray}
	&&dz_{1}(t,x)=\left[\Delta z_{1}(t,x)-\lambda_{11}z_{1}^{-2}(t,x)-\lambda_{12}z_{2}^{-2}(t,x) \right]dt-z_{1}(t,x)dN_{1}(t),x\in D, \; t>0,\label{zeq1}\\
	&&du_{2}(t,x)=\left[\Delta z_{2}(t,x)-\lambda_{21}z_{2}^{-2}(t,x)-\lambda_{22}z_{1}^{-2}(t,x) \right]dt-z_{2}(t,x)dN_{2}(t),x\in D, \; t>0,\label{zeq2}\\
	&&\frac{ \partial z_{i}(t,x)}{\partial \nu}+\beta z_{i}(t,x)=0, \ \ t>0, \ \ x\in \partial D, \ \ i=1,2,\label{zeq3}\\
			&&0 < z_{i}(0,x)=1-f_{i}(x):=g_{i}(x)\leq 1, \ \ x\in D, \ \ i=1,2.\label{zeq4}
	\end{eqnarray}
This reformulation simplifies the associated analysis. Henceforth, our focus shifts to analyzing the above system instead of \eqref{b1}--\eqref{b1d}.    
    
In this section, we establish the local-in-time existence of solutions to the system \eqref{zeq1}--\eqref{zeq4}. Before proceeding with the analysis, we first introduce a suitable notion of solutions for this system, along with the corresponding formulation for a related random reaction-diffusion system.

Indeed, we conisder the following transformation 
\begin{align}\label{ran1}
	v_{i}(t,x) = \exp\{N_i(t)\}z_{i}(t,x),\  i=1,2,
	\end{align}
	for $t \geq 0,\ x \in D.$ Such transformation is inspired by Doss-Sussmann transformation \cite{Aaraya2020} and are available in the literature only when the noise is either additive or linear multiplicative which also makes the study more interesting and challenging.

Then system \eqref{zeq1}--\eqref{zeq4} is reduced to the following random PDE system, see also Theorem \ref{thm2.1},
\bge
	&&\frac{\partial v_{i}(t,x)}{\partial t}=\left( \Delta-\frac{k_{i1}^{2}}{2} \right)v_{i}(t,x)-\lambda_{i1}v^{-2}_{i}(t,x)e^{3N_{i}(t)}
     -\lambda_{i2}v^{-2}_{j}(t,x)e^{N_{i}(t)+2N_{j}(t)}, \ \ t>0, \ \ x\in D, \qquad\label{rand1a}\\
	&&\frac{ \partial v_{i}(t,x)}{\partial \nu}+\beta v_{i}(t,x)=0, \ \ t>0, \ \ x\in \partial D,\label{rand1b} \\
	&&0 < v_{i}(0,x)=g_{i}(x)\leq 1, \ \ x\in D,\label{rand1c}
	\ege
for $i=1,2$ and $j\in\{1,2\}\setminus \{i\}.$

Let us recall the notion of weak and mild solutions  for systems \eqref{zeq1}--\eqref{zeq4} and \eqref{rand1a}--\eqref{rand1c}.

	\begin{Def}{\bf (Weak solutions)} 
    \begin{itemize}
        \item  A continuous $\{\mathcal{F}_{t}\}_{t\geq 0}$-adapted random field $z=\left\lbrace (z_{1}(t,x),z_2(t,x))^{\top}: \ 0\leq t\leq T,\ x \in D \right\rbrace$ is a \emph{weak solution} of system  \eqref{zeq1}--\eqref{zeq4} up to stopping time $\tau_q$  provided 
  \begin{itemize}
\item[(i)]
\begin{eqnarray*}
&&\int_0^t   \left[1+\langle \u_i(\cdot, s), \phi_i \rangle_{L^2(D)}\right]\,ds<\infty,\\&&  \langle \u_i(\cdot, *), \phi_i \rangle_{L^2(D)}\in \mathcal{C}^{\beta}[0,t]\;\mbox{for some}\; \beta>1-H,
\end{eqnarray*}
\item[(ii)] 
\begin{eqnarray*}\int_0^t  \left( \vert \langle \u_i(\cdot, s), \Delta\phi_i \rangle_{L^2(D)}\vert+\vert\langle \u^{-2}_i(\cdot, s) , \phi \rangle_{\DD}\vert+\vert\langle \u^{-2}_j(\cdot, s) , \phi \rangle_{L^2(D)}\vert\right)ds<\infty,
\end{eqnarray*}
\end{itemize}
and
\begin{itemize}
 \item[(iii)] 
 \begin{align} 
\int_{D}z_{i}\left(t,x\right)\varphi_{i}(x)dx&=\int_{D}g_{i}\left(x\right)\varphi_{i}(x)dx+\int_{0}^{t}\int_{D}z_{i}\left(s,x\right)\Delta \varphi_{i}(x)dxds\nonumber\\
&  -\lambda_{i1}\int_{0}^{t}\int_{D} z^{-2}_{i}\left(s,x\right)\varphi_{i}(x)dxds-\lambda_{i2}\int_{0}^{t}\int_{D} z^{-2}_{j}\left(s,x\right)\varphi_{i}(x)dxds\nonumber\\&-\int_{0}^{t}\int_{D}z_{i}\left(s,x\right)\varphi_{i}(x)dxdN_{i}(s),\ \mathbb{P}\mbox{-a.s.}\label{sk2}
	\end{align} 	
	hold for every $\varphi_i\in C^2(D),\ i=1,2,$ satisfying the boundary condition \eqref{a3a} and $\; j\in\{1,2\}\setminus\{i\}$ within  the time interval $\left(0, \min\{T,\tau_q\}\right).$ 
    \end{itemize}
    Conditions $(i)$ and $(ii)$ guarantee that the Itô, the fractional and the Lebesgue integrals in \eqref{sk2} are well defined, see also \cite{doz2023}.
 \item  A continuous $\{\mathcal{F}_{t}\}_{t\geq 0}$-adapted random field $v=\left\lbrace (v_{1}(t,x),v_2(t,x))^{\top}: \ 0\leq t\leq T,\ x \in D \right\rbrace$ is a \emph{weak solution} of random PDE system  \eqref{rand1a}--\eqref{rand1c} up to stopping time $\tau_q$  provided
 \begin{align} 
\int_{D}v_{i}\left(t,x\right)\varphi_{i}(x)dx&=\int_{D}g_{i}\left(x\right)\varphi_{i}(x)dx+\int_{0}^{t}\int_{D}v_{i}\left(s,x\right)\left( \Delta-\frac{k_{i1}^{2}}{2} \right) \varphi_{i}(x)dxds\nonumber\\
&  -\lambda_{i1}\int_{0}^{t}\int_{D} \exp\{{3N_{i}(t)}\} v^{-2}_{i}\left(s,x\right)\varphi_{i}(x)dxds\nonumber\\&
    -\lambda_{i2}\int_{0}^{t}\int_{D} \exp\{{N_{i}(t)+2 N_j(t)}\} v^{-2}_{j}\left(s,x\right)\varphi_{i}(x)dxds\label{sk3a}
	\end{align} 	
	hold for every $\varphi_i\in C^2(D),\ i=1,2,$ satisfying the boundary condition \eqref{a3a} and $\; j\in\{1,2\}\setminus\{i\}$ within  the time interval $\left(0, \min\{T,\tau_q\}\right).$ 
 
    \end{itemize} 
	\end{Def}
\newpage
	\begin{Def}{\bf (Mild solutions)}\label{df2}
    \begin{itemize}
    \item
	 A continuous $\{\mathcal{F}_{t}\}_{t \geq 0}$-adapted random field, $z=\left\lbrace (z_{1}(t,x),z_2(t,x))^{\top}: \ 0\leq t\leq T,\ x \in D \right\rbrace$ is a \emph{mild solution} of the system \eqref{zeq1}--\eqref{zeq4} up to $\tau_q$ if $z\in L^2(\Omega\times [0,T];\mathscr{H})\times L^2(\Omega\times [0,T];\mathscr{H})$ for  $\mathscr{H}:= W ^{2,2}(D)\cap W^{1,2}(D)$ and  satisfies 
	\begin{align}
	z_{i}(t,x)&=S_{t}g_{i}(x)-\int_{0}^{t} S_{t-r}[\lambda_{i1} z_{i}^{-2}(r,x)+\lambda_{i2}z_{j}^{-2}](r,x)dr-\int_{0}^{t}S_{t-r}z_{i}(r,x)dN_{i}(r),\ \mathbb{P}\mbox{-a.s.}, \nonumber
	\end{align}
	$i=1,2,\; j\in\{1,2\}\setminus\{i\},$ for  $x \in D$ and within  the time interval $\left(0, \min\{T,\tau_q\}\right).$
    \item A continuous $\{\mathcal{F}_{t}\}_{t\geq 0}$-adapted random field $v=\left\lbrace (v_{1}(t,x),v_2(t,x))^{\top}: \ 0\leq t\leq T,\ x \in D \right\rbrace$ with values in $L^2( [0,T];\mathscr{H})\times L^2( [0,T];\mathscr{H})$ is a \emph{mild solution} of random PDE system  \eqref{rand1a}--\eqref{rand1c} up to $\tau_q$ if
    \begin{align}\label{a2}
	v_{i}(t,x)&=\exp \left\{- \frac{k_{i1}^{2} t}{2} \right\}S_{t}g_{i}(x)\nonumber\\
    &\quad-\lambda_{i1}\int_{0}^{t} \exp\Big\{-\frac{k_{i1}^{2}}{2}(t-r) \Big\} S_{t-r}\left[ \exp\{3N_{i}(r)\}v^{-2}_{i}(r,\cdot) \right](x)dr, \nonumber\\
 & \quad-\lambda_{i2}\int_{0}^{t} \exp\Big\{-\frac{k_{i1}^{2}}{2}(t-r) \Big\} S_{t-r}\left[ \exp\{N_{i}(r)+2N_j(r)\}v^{-2}_{j}(r,\cdot) \right](x)dr, 
	\end{align}
	for  
	$i=1,2,\; j\in\{1,2\}\setminus\{i\}$ and for  $x \in D, t\in \left(0, \min\{T,\tau_q\}\right).$
\end{itemize}
	\end{Def} 
We now define the notion of quenching time for the above system as follows:
	\begin{Def}
	A stopping time $\tau_q: \Omega \rightarrow \mathbb{R}^{+}$ is called a quenching time of the system \eqref{zeq1}--\eqref{zeq4} if $$ \displaystyle\lim_{t \rightarrow \tau_q} \inf_{x \in D} \min\left\{\min_{x\in \bar{D}}|z_{1}(\cdot, t)|,\min_{x\in \bar{D}}|z_{2}(\cdot, t)|\right\} = 0,\ \mathbb{P}{\text-a.s.},$$ on the event  $\left\lbrace \omega \in \Omega, \tau_q(\omega)< \infty \right\rbrace.$
    The solution $z=(z_1,z_2)^{\top}$ of the system \eqref{zeq1}--\eqref{zeq4}  exist globally if $\tau_q=\infty,\ \mathbb{P}{\text-a.s.}$
	\end{Def}
   

	
Next we turn to establishing the connection between the weak formulations \eqref{sk2} and \eqref{sk3a}, thereby linking the corresponding weak solutions.

\begin{theorem} \label{thm2.1}
	Let $z=(z_1,z_2)^{\top}$ be a weak solution of the system \eqref{zeq1}--\eqref{zeq4}. Then the function $v=(v_1,v_2)^{\top}$ defined by \eqref{ran1}
	is a weak solution of the system of random PDE system \eqref{rand1a}--\eqref{rand1c} and viceversa. 
	\begin{proof}
	Applying Itô's formula (see Theorem \ref{ito2}), we obtain  
\begin{equation*}
\exp\{N_i(t)\} = 1 + \int_0^t \exp\{N_i(s)\} \, dN_i(s) + \frac{k_{i1}^2}{2} \int_0^t \exp\{N_i(s)\} \, ds,
\end{equation*}
taking also into account the initial condition \( N_i(0) = 1 \).

In differential form, this can be written as
\begin{align}
d\left(\exp\{N_i(t)\}\right) &= \exp\{N_i(t)\} \, dN_i(t) + \frac{k_{i1}^2}{2} \exp\{N_i(t)\} \, dt, \quad t > 0, \label{ifo1} \\
\exp\{N_i(0)\} &= 1, \label{ifo2}
\end{align}
for \( i = 1, 2 \).
	
For any  function $\varphi_{i}\in C^2(D),\ i=1,2$, satisfying the boundary condition \eqref{a3a}, we set
	$$z_{i}(t,\varphi_{i})=\int_{D}z_{i}(t,x) \varphi_{i}(x)dx.$$
	Then the weak formulation \eqref{sk2} implies 
	\begin{align}\label{c1} 	z_{i}\left(t,\varphi_{i}\right)&=z_{i}\left(0,\varphi_{i}\right)+\int_{0}^{t}z_{i}\left(s,\Delta \varphi_{i}\right)ds -\lambda_{i1}\int_{0}^{t} z^{-2}_{i}\left(s,\varphi_{i}\right)ds-\lambda_{i2}\int_{0}^{t} z^{-2}_{j}\left(s,\varphi_{i}\right)ds\nonumber\\
	&\quad-\int_{0}^{t}z_{i}\left(s,\varphi_{i}\right)dN_{i}(s),\ \mathbb{P}\mbox{-a.s.},
	\end{align}
	for $i=1,2$ and $j\in\{1,2\}\setminus \{i\}.$ 

Therefore, by applying the integration by parts formula \eqref{int2} (see also \cite{biagini}), and taking into account \eqref{ran1}, we derive the following representation for  
$
v_i(t, \varphi_i) := \int_{D} v_i(t, x) \varphi_i(x) \, dx,\, i = 1, 2:
$

	\begin{align*}
	v_{i}(t,\varphi_{i})&=v_{i}(0,\varphi_{i})+\int_{0}^{t} \exp\{{N_{i}(s)}\}dz_{i}(s,\varphi_{i})+\int_{0}^{t} z_{i}(s,\varphi_{i})d\Big(\exp\{{N_{i}(s)}\} \Big) \nonumber\\&\quad + \Big[dz_{i}(s,\varphi_{i}), d\Big(\exp\{{N_{i}(s)}\} \Big)\Big],
	\end{align*} 
    where 
     $$\Big[dz_{i}(s,\varphi_{i}), d\Big(\exp\{{N_{i}(s)}\} \Big)\Big]=-k_{i1}^{2}\int_{0}^{t}\exp\{{N_{i}(s)}\} z_{i}(t, \varphi_{i})ds,\ i=1,2.$$
As a consequence of equations \eqref{ifo1}--\eqref{ifo2}, we obtain
	\begin{align*}
	v_{i}(t,\varphi_{i})
    &=v_{i}(0,\varphi_{i})+\int_{0}^{t} \exp\{N_i(s)\}dz_{i}(s,\varphi_{i})+\int_{0}^{t} z_{i}(s,\varphi_{i})\Bigg(\exp\{{N_{i}(s)}\} dN_{i}(s)\nonumber\\&\quad+\frac{k_{i1}^{2}}{2}\int_{0}^{t} \exp\{{N_{i}(s)}\}ds\Bigg)-k_{i1}^{2}\int_{0}^{t}\exp\{{N_{i}(s)}\} z_{i}(t, \varphi_{i})ds,
	\end{align*} 
	for $i=1,2.$

   Combining \eqref{zeq1}, \eqref{ran1}, and \eqref{c1}, we thus arrive at
    \begin{align} \label{qq1}
    v_{i}(t,\varphi_{i}) &= v_{i}(0,\varphi_{i})+\int_{0}^{t}\Delta v_{i}\left(s, \varphi_{i}\right)ds-\lambda_{i1}\int_{0}^{t}v^{-2}_{i}\left(s,\varphi_{i}\right) \exp\{{3N_{i}(s)}\}ds\nonumber\\
	&\quad-\lambda_{i2}\int_{0}^{t}v^{-2}_{j}\left(s,\varphi_{i}\right) \exp\{{N_{i}(s)+2N_{j}(s)}\}ds\nonumber\\
    &\quad-\frac{k_{i1}^{2}}{2} \int_{0}^{t}v_{i}(s, \varphi_{i})ds,
    \end{align}
    for $i=1,2$ and $j\in\{1,2\}\setminus \{i\}.$  
 
 It thus follows from the preceding relation that \( v = (v_1, v_2)^{\top} \) is a weak solution of the random PDE system \eqref{rand1a}--\eqref{rand1c}.

The converse implication holds due to the change of variables being implemented via a homeomorphism, thereby transforming one random dynamical system into an equivalent counterpart.
	
	\end{proof}
	\end{theorem}
 \begin{Rem}
    Let \( \tau_q \) denote the quenching time of the system \eqref{rand1a}--\eqref{rand1c} corresponding to initial data of the specified form. By Theorem~\ref{thm2.1}, together with the almost sure continuity of the processes \( W(\cdot) \) and \( B^{H}(\cdot) \), it follows that \( \tau_q \) also serves as the quenching time for the coupled system \eqref{zeq1}--\eqref{zeq4}. One of our objectives is to identify random times \( \tau_* \) and \( \tau^* \) such that
\(
0 < \tau_* \leq \tau_q \leq \tau^*,
\)
which provide lower and upper bounds for the quenching time \( \tau_q \) almost surely.

\end{Rem}
 The equivalence between weak and mild solutions of  system \eqref{rand1a}--\eqref{rand1c} can be established in a similar way as in \cite[Theorem 2.2]{doz2023}. Indeed, the following holds.
\begin{proposition}\label{equ1}
If \( v = (v_1, v_2)^{\top} \) is a weak solution of the system \eqref{rand1a}--\eqref{rand1c}, then \( v = (v_1, v_2)^{\top} \) is also a mild solution of the system, and vice versa.
\end{proposition}
The following theorem provides a local-in-time  existence result for a mild  solution $v=(v_1,v_2)^{\top}$ for
	the system of random PDEs \eqref{rand1a}--\eqref{rand1c}, which is also a weak solution thanks to the result of Proposition \ref{equ1}.
	\begin{theorem}
	There exists $\tau_q>0$ such that system \eqref{rand1a}--\eqref{rand1c} has a unique mild solution $v=(v_1,v_2)^{\top}$ such that $v_{i}$ in $L^{\infty}([0,\tau_{q}) \times D ),\ i=1,2,$ $\mathbb{P}$-a.s in $\Omega.$ 
	\end{theorem}
	\begin{proof}
	 For the reader’s convenience, we provide a complete proof below, following the same steps as in \cite[Proposition 3.8]{KNY24}.

According to Definition \ref{df2} we need to prove that there is $\tau_q>0$ such that
\begin{align*}
	v_{i}(t,x)&=\exp \left\{- \frac{k_{i1}^{2} t}{2} \right\}S_{t}g_{i}(x)\nonumber\\
    &\quad-\lambda_{i1}\int_{0}^{t} \exp\Big\{-\frac{k_{i1}^{2}}{2}(t-r) \Big\} S_{t-r}\left[ \exp\{3N_{i}(r)\}h(v_{i}(r,\cdot)) \right](x)dr, \nonumber\\
 & \quad-\lambda_{i2}\int_{0}^{t} \exp\Big\{-\frac{k_{i1}^{2}}{2}(t-r) \Big\} S_{t-r}\left[ \exp\{N_{i}(r)+2N_j(r)\}h(v_{j}(r,\cdot)) \right](x)dr, 
	\end{align*}
for  $i=1,2,\; j\in\{1,2\}\setminus\{i\}$ and for each $t \in(0, \tau_q)$ and $x \in D$.

To address the singular behaviour of the function \( h(s) = s^{-2} \) near \( s = 0 \), we introduce a sequence of approximating functions \( (h_n)_{n \in \mathbb{N}} \) defined by  
\[
h_n(s) := \left( \max\left\{ s, \frac{1}{n} \right\} \right)^{-2}, \quad \text{for } s > 0,\quad n = 1, 2, 3, \dots
\]
This construction regularizes the singularity at the origin by effectively truncating \( h(s) \) below  \( s = \frac{1}{n} \). Since \( h(s) \) is strictly decreasing in a neighborhood of \( 0 \), it follows that \( -h_n(s) \) is bounded below by \( -h\left( \frac{1}{n} \right) \).

Moreover, each \( h_n(s) \) inherits the local Lipschitz continuity of \( h(s) \), as the truncation preserves smoothness away from the singularity. Let \( C_n \) denote the uniform bound of \( |h_n(s)| \), and for any fixed \( R > 0 \), let \( L_n(R) \) represent the Lipschitz constant of \( h_n(s) \) on the interval \( (0, R) \).

It is important to note that for any \( v_1, v_2 > 0 \), we have \( h_n(v_i) = h(v_i) \) for \( i = 1, 2 \) whenever \( \min\{v_1, v_2\} > \frac{1}{n} \). Additionally, the sequence \( (h_n)_{n \in \mathbb{N}} \) is monotonic in the sense that if \( n < m \), then
\[
h_n(v_i) \leq h_m(v_i), \quad \text{or equivalently,} \quad -h_n(v_i) \geq -h_m(v_i), \quad \text{for } i = 1, 2.
\]
Next, we consider the random field \( v^{(n)} = (v_1^{(n)}, v_2^{(n)})^{\top} \) as the mild solution to the approximate system
\begin{align}\label{cd1}
v_i^{(n)}(t,x) &= e^{-\frac{k_{i1}^2 t}{2}} S_t g_i(x) 
- \lambda_{i1} \int_0^t e^{-(t-r)\frac{k_{i1}^2}{2}} S_{t-r} \left[ 
e^{N_i(r)} h_n\left( e^{-N_i(r)} v_i^{(n)}(r, \cdot) \right) \right](x) \, dr \nonumber \\
&\quad - \lambda_{i2} \int_0^t e^{-(t-r)\frac{k_{i1}^2}{2}} S_{t-r} \left[ 
e^{N_i(r)} h_n\left( e^{-N_j(r)} v_j^{(n)}(r, \cdot) \right) \right](x) \, dr,
\end{align}
for \( i = 1, 2 \), where \( j \in \{1, 2\} \setminus \{i\} \).

To ensure the well-posedness of the system \eqref{cd1} within a controlled regime, we define the stopping time
\begin{equation}\label{ik37}
\tau_n := \min\{ \tilde{\tau}_{n}, T_n \},
\end{equation}
where
\[
\tilde{\tau}_{n} := \min\{ \tilde{\tau}_1^{(n)}, \tilde{\tau}_2^{(n)} \},
\]
and
\[
\tilde{\tau}_i^{(n)} := \inf \left\{ t \geq 0 \ : \ \inf_{x \in D} v_i^{(n)}(t,x) \leq \frac{1}{n} \right\}, \; i=1,2,
\]
while $T_n$ will be detemined below.

This stopping mechanism ensures that the solution remains uniformly bounded away from the singularity of \( h(s) = s^{-2} \), thereby guaranteeing the regularity of the approximated dynamics over the time interval \( [0, \tau_n] \).

We aim to show that the system \eqref{cd1} admits a unique solution in \( L^{\infty}([0, \tau_n) \times D) \) for sufficiently small \( T_n > 0 \), for every \( n \in \mathbb{N} \), ensuring that the stopping times \( \tau_n \) are well-defined. This follows from the local Lipschitz continuity of the approximating functions \( h_n \), as demonstrated below.

Observe that the sequence of stopping times satisfies \( \tau_n \leq \tau_m \) whenever \( n < m \). Indeed, if \( \min\{v_1, v_2\} > \frac{1}{n} \), then \( h_n(v_i) = h_m(v_i) \) for \( i = 1, 2 \). By the uniqueness of the mild solution \ to system \eqref{cd1}, it follows that \( v_i^{(n)} = v_i^{(m)} \), \( i = 1, 2 \), as long as both remain above \( \frac{1}{n} \), i.e., for all \( t \leq \tau_n \). 

Therefore, up to time \( \tau_n \), neither \( v_i^{(n)} \) nor \( v_i^{(m)} \), for \( i = 1, 2 \), fall below \( \frac{1}{n} \), implying that \( v_i^{(m)} \) can only drop below \( \frac{1}{m} < \frac{1}{n} \) at some time \( \tau_m \geq \tau_n \). Consequently, the sequence \( (\tau_n)_{n \in \mathbb{N}} \) forms an increasing and bounded (by \( \sup_{n \in \mathbb{N}} \tau_n  \)) sequence of stopping times.

Since \( h_n = h \) on \( [0, \tau_n) \), it follows that \( v_i^{(n)} = v_i \) for \( i = 1, 2 \) on \( [0, \tau_n) \), where \( v_i \) denotes the solution to the system \eqref{cd1}. Accordingly, we define the  solution \( v = (v_1, v_2)^\top \) of system \eqref{a2} via
\begin{equation}\label{sk46}
v_i(t, x) := v_i^{(n)}(t, x), \quad \forall n \geq 1,\quad (t, x) \in [0, \tau_n) \times D.
\end{equation}
This construction yields a solution defined on \( [0, \tau_q) \times D \), where
\begin{equation}\label{nk71a}
\tau_q := \sup_{n \in \mathbb{N}} \tau_n > 0,
\end{equation}
which serves as the local solution to the system \eqref{cd1}.

We now establish the existence of a local-in-time mild solution \( v^{(n)} = (v_1^{(n)}, v_2^{(n)})^{\top} \) to the system \eqref{cd1} in the space \( L^{\infty}([0, T_n) \times D) \), for sufficiently small \( T_n > 0 \). To this end, we employ a fixed-point argument for the operators \( \mathcal{T}_1 \) and \( \mathcal{T}_2 \), defined respectively by
\begin{align*}
\mathcal{T}_1[v_1^{(n)}, v_2^{(n)}](t,x) 
&= e^{- \frac{k_{11}^2 t}{2}} S_t g_1(x) 
- \lambda_{11} \int_0^t e^{- \frac{k_{11}^2}{2}(t - r) + N_1(r)} 
\left[ S_{t-r} h_n\left( e^{-N_1(r)} v_1^{(n)}(r, \cdot) \right) \right](x) \, dr \\
&\quad - \lambda_{12} \int_0^t e^{- \frac{k_{11}^2}{2}(t - r) + N_1(r)} 
\left[ S_{t-r} h_n\left( e^{-N_2(r)} v_2^{(n)}(r, \cdot) \right) \right](x) \, dr,
\end{align*}
and
\begin{align*}
\mathcal{T}_2[v_1^{(n)}, v_2^{(n)}](t,x) 
&= e^{- \frac{k_{21}^2 t}{2}} S_t g_2(x) 
- \lambda_{21} \int_0^t e^{- \frac{k_{21}^2}{2}(t - r) + N_2(r)} 
\left[ S_{t-r} h_n\left( e^{-N_2(r)} v_2^{(n)}(r, \cdot) \right) \right](x) \, dr \\
&\quad - \lambda_{22} \int_0^t e^{- \frac{k_{21}^2}{2}(t - r) + N_2(r)} 
\left[ S_{t-r} h_n\left( e^{-N_1(r)} v_1^{(n)}(r, \cdot) \right) \right](x) \, dr.
\end{align*}
These operators are defined on the Banach space \( X \times X \), where
\[
X := \left\{ w: [0, T_n] \times D \rightarrow \mathbb{R} \ \big| \ \|w\|_* < \infty \right\},
\]
equipped with the norm
\[
\|w\|_* := \sup_{t \in [0, T_n]} \|w(t, \cdot)\|_{\infty}.
\]
Here, \( T_n > 0 \) is a suitably small time to be determined so as to ensure that the mapping \( (\mathcal{T}_1, \mathcal{T}_2) \) defines a contraction on \( X \times X \), guaranteeing the existence and uniqueness of a fixed point, i.e., the mild solution \( v^{(n)} \).

	We will show that the operators $\mathcal{T}_{1}$ and $\mathcal{T}_{2}$ leave invariant and are contractions on the subset $X_R\times X_R$ for $X_{R}:=\{w \in X \ | \ w \geq 0, \|w\|_{\ast}<R\}$ for a suitable $R$ hence guaranteeing the existence of a unique non negative solution.
	
We first consider the invariance property within $X_{R}\times X_{R} .$ Consider any $v^{(n)}_{i} \in X_{R},\ i=1,2$, we have that
	\begin{align*}
	\|\mathcal{T}_1[v_1^{(n)}, v_2^{(n)}]\|_{\ast}&=\sup_{t \in [0,T_{n}]} \Big\|e^{- \frac{k_{11}^{2} t}{2}}S_{t}g_{1}(\cdot)-\lambda_{11}\int_{0}^{t} e^{-\frac{k_{11}^{2}}{2}(t-r)+N_{1}(r)} \left[S_{t-r}h_{n}\left(e^{-N_{1}(r)}v_{1}^{(n)}(r,\cdot)\right) \right] (x)dr\nonumber\\
	& \quad -\lambda_{12}\int_{0}^{t} e^{-\frac{k_{11}^{2}}{2}(t-r)+N_{1}(r)}\left[  S_{t-r}h_{n}\left(e^{-N_{2}(r)}v_{2}^{(n)}(r,\cdot)\right)\right] (x)dr \Big\|_{\infty} \nonumber\\
	&\leq \Big\|e^{- \frac{k_{11}^{2} T_n}{2}}S_{t}g_{1}(\cdot)\Big\| _{\infty}+\lambda_{11}\sup_{t \in [0,T_{n}]} \int_{0}^{t}e^{-\frac{k_{11}^{2}}{2}(t-r)+N_{1}(r)}\Big\| S_{t-r}h_{n}\left(e^{-N_{1}(r)}v_{1}^{(n)}(r,\cdot)\right) (x)\Big\|_{\infty}dr  \nonumber\\
	& \quad +\lambda_{12}\sup_{t \in [0,T_{n}]} \int_{0}^{t} e^{-\frac{k_{11}^{2}}{2}(t-r)+N_{1}(r)}\Big\| S_{t-r}h_{n}\left(e^{-N_{2}(r)}v_{2}^{(n)}(r,\cdot)\right) (x)\Big\|_{\infty}dr  \nonumber\\
	&\leq \Big\|e^{- \frac{k_{11}^{2} T_{n}}{2}}S_{t}g_{1}(\cdot)\Big\| _{\infty}+\lambda_{11}\sup_{t \in [0,T_{n}]} \int_{0}^{t}e^{N_{1}(r)}\Big\| S_{t-r}h_{n}\left(e^{-N_{1}(r)}v_{1}^{(n)}(r,\cdot)\right) (x)\Big\|_{\infty}dr  \nonumber\\
	& \quad +\lambda_{12}\sup_{t \in [0,T_{n}]} \int_{0}^{t} e^{N_{1}(r)}\Big\| S_{t-r}h_{n}\left(e^{-N_{2}(r)}v_{2}^{(n)}(r,\cdot)\right) (x)\Big\|_{\infty}dr  \nonumber\\
	&\leq \Big\|e^{- \frac{k_{11}^{2} T_{n}}{2}}S_{t}g_{1}(\cdot)\Big\| _{\infty}+\lambda_{11}\sup_{t \in [0,T_{n}]} \int_{0}^{t}e^{N_{1}(r)}\Big\| h_{n}\left(e^{-N_{1}(r)}v_{1}^{(n)}(r,\cdot)\right) (x)\Big\|_{\infty}dr  \nonumber\\
	& \quad +\lambda_{12}\sup_{t \in [0,T_{n}]} \int_{0}^{t} e^{N_{1}(r)}\Big\| h_{n}\left(e^{-N_{2}(r)}v_{2}^{(n)}(r,\cdot)\right) (x)\Big\|_{\infty}dr  \nonumber\\		
	&\leq \Big\|e^{- \frac{k_{11}^{2} T_{n}}{2}}S_{t}g_{1}(\cdot)\Big\| _{\infty}+\lambda_{11}T_{n}A^{T_{n}}_{1} C_{n}+\lambda_{12}T_{n}A^{T_{n}}_{1} C_{n}\nonumber\\
	&\leq e^{- \frac{k_{11}^{2} T_{n}}{2}}\|g_{1}\| _{\infty}+\left(\lambda_{11}T_{n}A^{T_{n}}_{1} +\lambda_{12}T_{n}A^{T_{n}}_{1} \right)C^{\ast}\nonumber\\
    &\leq \|g_{1}\| _{\infty}+\left(\lambda_{11}T_{n}A^{T_{n}}_{1} +\lambda_{12}T_{n}A^{T_{n}}_{1} \right)C^{\ast}
	\end{align*}
	where $A^{T_{n}}_{1}=\displaystyle\sup_{t \in[0, T_{n}]}e^{|N_{1}(t)|},$  $C^{\ast}=\displaystyle\sup_{n \in \mathbb{N}} C_{n}$. Note that since $v_{i}^{n}(x, 0)=g_{i}(x)>0,\ i=1,2$, for $n=1,2,\dots$ then by continuity argument we have $C^{\ast}>0,\ i=1,2.$  
	
We then choose $R_{1}$ and $T_{n}$ such that
	\begin{align} 
	\|g_{1}\| _{\infty}+\left(\lambda_{11}T_{n}A^{T_{n}}_{1} +\lambda_{12}T_{n}A^{T_{n}}_{1} \right)C^{\ast}<R_{1}. \nonumber 
	\end{align}
	Similarly, we can choose $R_{2}$ and $T_{n}$ such that 
	\begin{align} 
	\|\mathcal{T}_{2}[v_{2}^{n},v_{1}^{n}]\|_{\ast} \leq	\|g_{2}\| _{\infty}+\left(\lambda_{21}T_{n}A^{T_{n}}_{2} +\lambda_{22}T_{n}A^{T_{n}}_{2}\right)C^{\ast}<R_{2}, \nonumber 
	\end{align}
	where $A^{T_{n}}_{2}=\displaystyle\sup_{t \in[0, T_{n}]}e^{|N_{2}(t)|}.$ Taking $R=\max\{R_1,R_2\}>0$ we obtain that  $X_{R}\times X_{R}$ is invariant under the maping  $(\mathcal{T}_{1},\mathcal{T}_{2}).$
	
We now consider the contraction property. Indeed, for $(v_{1}^{(n)},v_{2}^{(n)}), (w_{1}^{(n)},w_{2}^{(n)})\in X_R\times X_R$ 
	\begin{align}\label{co1}
	\|\mathcal{T}_{1}[v_{1}^{(n)},v_{2}^{(n)}]-\mathcal{T}_{1}[w_{1}^{(n)},w_{2}^{(n)}]\|_{\ast}	&=\sup_{t \in [0,T_n]} \Big\|	-\lambda_{11}\int_{0}^{t} e^{-\frac{k_{11}^{2}}{2}(t-r)+N_{1}(r)} \left[S_{t-r}h_{n}\left(e^{-N_{1}(r)}v_{1}^{(n)}(r,\cdot)\right) \right] (x)dr\nonumber\\
	&\quad-\lambda_{12}\int_{0}^{t} e^{-\frac{k_{11}^{2}}{2}(t-r)+N_{1}(r)}\left[  S_{t-r}h_{n}\left(e^{-N_{2}(r)}v_{2}^{(n)}(r,\cdot)\right)\right] (x)dr\nonumber\\
	&\quad+\lambda_{11}\int_{0}^{t} e^{-\frac{k_{11}^{2}}{2}(t-r)+N_{1}(r)} \left[S_{t-r}h_{n}\left(e^{-N_{1}(r)}w_{1}^{(n)}(r,\cdot)\right) \right] (x)dr\nonumber\\ 
	&\quad+\lambda_{12}\int_{0}^{t} e^{-\frac{k_{11}^{2}}{2}(t-r)+N_{1}(r)}\left[  S_{t-r}h_{n}\left(e^{-N_{2}(r)}w_{2}^{(n)}(r,\cdot)\right)\right] (x)dr\Big\|_{\infty}, \nonumber\\    
	&\leq \lambda_{11}\sup_{t \in T_n} \int_{0}^{t} e^{-\frac{k_{11}^{2}}{2}(t-r)+N_{1}(r)} \nonumber\\
	&\qquad \times\Big\| S_{t-r}h_{n}\left(e^{-N_{1}(r)}v_{1}^{(n)}(r,\cdot)\right)-S_{t-r}h_{n}\left(e^{-N_{1}(r)}w_{1}^{(n)}(r,\cdot)\right)  \Big\|_{\infty} \nonumber\\
	&\quad+\lambda_{12}\sup_{t \in T_n} \int_{0}^{t} e^{-\frac{k_{11}^{2}}{2}(t-r)+N_{1}(r)} \nonumber\\
	&\qquad \times\Big\|   S_{t-r}h_{n}\left(e^{-N_{2}(r)}v_{2}^{(n)}(r,\cdot)\right)-  S_{t-r}h_{n}\left(e^{-N_{2}(r)}w_{2}^{(n)}(r,\cdot)\right)  \Big\|_{\infty} \nonumber\\
	&\leq \lambda_{11}\sup_{t \in T_n} \int_{0}^{t} e^{N_{1}(r)} \nonumber\\
	&\qquad \times\Big\| S_{t-r}h_{n}\left(e^{-N_{1}(r)}v_{1}^{(n)}(r,\cdot)\right)-S_{t-r}h_{n}\left(e^{-N_{1}(r)}w_{1}^{(n)}(r,\cdot)\right)  \Big\|_{\infty} \nonumber\\
	&\quad+\lambda_{12}\sup_{t \in T_n} \int_{0}^{t} e^{N_{1}(r)} \nonumber\\
	&\qquad \times\Big\|   S_{t-r}h_{n}\left(e^{-N_{2}(r)}v_{2}^{(n)}(r,\cdot)\right)-  S_{t-r}h_{n}\left(e^{-N_{2}(r)}w_{2}^{(n)}(r,\cdot)\right)  \Big\|_{\infty} \nonumber\\
	&\leq \lambda_{11}\sup_{t \in T_n} \int_{0}^{t} e^{N_{1}(r)} \nonumber\\
	&\qquad \times\Big\| h_{n}\left(e^{-N_{1}(r)}v_{1}^{(n)}(r,\cdot)\right)-h_{n}\left(e^{-N_{1}(r)}w_{1}^{(n)}(r,\cdot)\right)  \Big\|_{\infty} \nonumber\\
	&\quad+\lambda_{12}\sup_{t \in T_n} \int_{0}^{t} e^{N_{1}(r)} \nonumber\\
	&\qquad \times\Big\| h_{n}\left(e^{-N_{2}(r)}v_{2}^{(n)}(r,\cdot)\right)-  h_{n}\left(e^{-N_{2}(r)}w_{2}^{(n)}(r,\cdot)\right)  \Big\|_{\infty},
	\end{align}	
	where we first used the fact that $e^{-\frac{k_{i1}^{2}}{2}(t-r)}<1,\ i=1,2$ and then the contractivity of the evolution family of operator $S_{t}.$ To proceed further we must make use of the local Lipschitz properties of the nonlinearities $h_{n}.$ Let $h_{n}$ be locally Lipschitz, satisfying the property
	\begin{align*}
	|h_{n}(z)-h_{n}(y)| \leq L_{n}(R) |z-y|,\ z, y \in (0, R),
	\end{align*} 
	for some $L_{n}(R)>0.$ Then as long as $|e^{-N_{i}(r)}| \  \|v_{i}^{(n)}(r,\cdot)\|_{\infty}<R$ and $|e^{-N_{i}(r)}| \  \|w_{i}^{(n)}(r,\cdot)\|_{\infty}<R,\ i=1,2,$ we may estimate
	\begin{align}\label{g2}
	 &\Big\|h_{n}\left(e^{-N_{1}(r)}v_{1}^{(n)}(r,\cdot)\right)-h_{n}\left(e^{-N_{1}(r)}w_{1}^{(n)}(r,\cdot)\right)\Big\|_{\infty} \nonumber\\
	 &\qquad\quad= \sup_{x \in D} \left|h_{n}\left(e^{-N_{1}(r)}v_{1}^{(n)}(r,x)\right)-h_{n}\left(e^{-N_{1}(r)}w_{1}^{(n)}(r,x)\right) \right| \nonumber\\
	 &\qquad\quad\leq L_{n}(R) e^{-N_{1}(r)} \sup_{x \in D} \left|v_{1}^{(n)}(r,x)-w_{1}^{(n)}(r,x) \right|  \nonumber\\
	 &\qquad\quad\leq L_{n}(R) e^{-N_{1}(r)} \left\|v_{1}^{(n)}(r,\cdot)-w_{1}^{(n)}(r,\cdot) \right\|_{\infty}. 
	\end{align} 
	Similarly, we have
	\begin{align}\label{g3}
	&\Big\|h_{n}\left(e^{-N_{2}(r)}v_{2}(r,\cdot)\right)- h_{n}\left(e^{-N_{2}(r)}w_{2}(r,\cdot)\right)  \Big\|_{\infty}\nonumber\\
	&\qquad\quad\leq L_{n}(R) e^{-N_{2}(r)} \left\|v_{2}(r, \cdot)-w_{2}(r, \cdot) \right\|_{\infty}. 
	\end{align}
	 By using \eqref{g2} and \eqref{g3} in \eqref{co1}, we obtain
	\begin{align*}
	&\|\mathcal{T}_{1}[v_{1}^{n},v_{2}^{n}]-\mathcal{T}_{1}[w_{1}^{n},w_{2}^{n}]\|_{\ast} \nonumber\\ &\qquad\leq \lambda_{11}\sup_{t \in [0,T_n]} \int_{0}^{t} e^{N_{1}(r)} L_{n}(R) e^{-N_{1}(r)} \left\|v_{1}(r, \cdot)-w_{1}(r, \cdot) \right\|_{\infty} \nonumber\\
	&\qquad\quad+\lambda_{12}\sup_{t \in [0,T_n]} \int_{0}^{t} e^{N_{1}(r)}  L_{n}(R) e^{-N_{2}(r)} \left\|v_{2}(r, \cdot)-w_{2}(r, \cdot) \right\|_{\infty}\nonumber\\
	&\qquad\leq \lambda_{11}\sup_{t \in [0,T_n]} \int_{0}^{t} L_{n}(R)  \left\|v_{1}(r, \cdot)-w_{1}(r, \cdot) \right\|_{\infty} \nonumber\\
	&\qquad\quad+\lambda_{12} A_1^{T_{n}}A_2^{T_{n}}\sup_{t \in [0,T_n]} \int_{0}^{t} L_{n}(R) \left\|v_{2}(r, \cdot)-w_{2}(r, \cdot) \right\|_{\infty}\nonumber\\
	&\qquad\leq \lambda_{11}T_{n} L_{n}(R)  \sup_{t \in [0,T_n]}\left\|v_{1}(r, \cdot)-w_{1}(r, \cdot) \right\|_{\infty} \nonumber\\
	&\qquad\quad+\lambda_{12}T_{n} A_1^{T_{n}}A_2^{T_{n}}L_{n}(R)\sup_{t \in [0,T_n]}   \left\|v_{2}(r, \cdot)-w_{2}(r, \cdot) \right\|_{\infty}\nonumber\\
	&\qquad= \lambda_{11}T_{n} L_{n}(R)  \left\|v_{1}(r, \cdot)-w_{1}(r, \cdot) \right\|_{\ast} \nonumber\\
	&\qquad\quad+\lambda_{12}T_{n} A_1^{T_{n}}A_2^{T_{n}}L_{n}(R)   \left\|v_{2}(r, \cdot)-w_{2}(r, \cdot) \right\|_{\ast}\nonumber\\
	&\qquad\leq T_{n}M_{n}\left(\left\|v_{1}(r, \cdot)-w_{1}(r, \cdot) \right\|_{\ast}+\left\|v_{2}(r, \cdot)-w_{2}(r, \cdot) \right\|_{\ast} \right),\nonumber
	\end{align*}
	where $M_{n}=\max\{ \lambda_{11} L_{n}(R), \lambda_{12} A_1^{T_{n}}A_2^{T_{n}}L_{n}(R)\}.$ Thus, we deduce that
	\begin{align*}
	\|\mathcal{T}_{1}[v_{1}^{n},v_{2}^{n}]-\mathcal{T}_{1}[w_{1}^{n},w_{2}^{n}]\|_{\ast} \leq T_nM_{n}\left(\left\|v_{1}(x,r)-w_{1}(x,r) \right\|_{\ast}+\left\|v_{2}(x,r)-w_{2}(x,r) \right\|_{\ast} \right),\nonumber 
	\end{align*}
	as long as $\displaystyle\sup_{t \in [0,T_n]}|e^{-N_{i}(t)}| \  \|v_{i}^{(n)}(t,\cdot)\|_{\infty}<R$ and $\displaystyle\sup_{t \in [0,T_n]}|e^{-N_{i}(t)}| \  \|w_{i}^{(n)}(t,\cdot)\|_{\infty}<R,\ i=1,2.$
Using the above reasoning we obtain a similar estimate for the operator $\mathcal{T}_{2}.$
    
	Hence, choosing $T_n$ so that $$T_{n}M_{n}<1,$$ it can be seen that the map $(\mathcal{T}_{1}, \mathcal{T}_{2}): X_R\times X_R \mapsto X_R\times X_R$ is a contraction, so it has a fixed point which is the unique solution to \eqref{cd1} in the interval $(0,\tau_n,)$
	where $\tau_n$ is defined by \eqref{ik37}. Then by virtue of \eqref{sk46}  there exists $\tau_q>0,$ cf. \eqref{nk71a}, such that \eqref{rand1a}--\eqref{rand1c} has a unique solution in $L^{\infty}([0, \tau_q) \times D )\times L^{\infty}([0, \tau_q) \times D ).$
	\end{proof}
    \begin{Rem}
Local existence of a unique  weak solution for system \eqref{zeq1}--\eqref{zeq4} arises as  an imediate consequence of Theorem \ref{thm2.1} and Proposition \ref{equ1}.
    \end{Rem}
    \section{Estimates of the quenching time and quenching rate}
    The primary objective of this section is to establish bounds for the quenching time \( \tau_q \) and the quenching rate of the solution to system~\eqref{rand1a}--\eqref{rand1c}, which can lead to analogous bounds for the corresponding stochastic system~\eqref{zeq1}--\eqref{zeq4}.

	\subsection{Global existence -- A lower bound for the quenching time}\label{sec3}
We begin by establishing a lower bound $\tau_*$ for the quenching (stopping)  time  $\tau_q.$ Given that weak and mild solutions are equivalent for problem \eqref{rand1a}--\eqref{rand1c} (cf. Proposition \ref{equ1}), we will work within the framework of mild solutions to derive the desired lower bounds. To achieve this, we adopt the strategy developed in \cite{DNMK23, KNY24}, following their methodology closely in the subsequent analysis.

We first conider the stochastic processes:
	\begin{equation}\label{psk1}
	\mathscr{G}_{1}(t)=\left[ 1-4(\lambda_{11}+\lambda_{12})\int_{0}^{t}  \max \Big\{\exp\{3N_{1}(r)\}, \exp\{N_{1}(r)+2N_{2}(s)\}\Big\}\mu_{1}^{-3}(r) dr\right]^{\frac{1}{4}}
    \end{equation}
   and
\begin{equation}\label{psk2}
	\mathscr{G}_{2}(t)=\left[ 1-4 (\lambda_{21}+\lambda_{22}) \int_{0}^{t}  \max\Big\{\exp\{3N_{2}(r)\}, \exp\{N_{2}(r)+2N_{1}(r)\}\Big\}\mu_{2}^{-3}(r) dr\right]^{\frac{1}{4}},
	\end{equation}
 where  
 \bge\label{46sk}
 \mu_{i}(t):= \displaystyle  \exp\Big\{{-\frac{k_{i1}^{2}}{2} t}\Big\} \inf_{x \in D}S_{t} g_{i}(x)>0, \quad i=1,2. 
 \ege
We now define the stopping time 
$\tau_*$
   \begin{align}\label{ST1}
	\tau_{\ast} = \inf \Bigg\{ t\geq 0 : &\int_{0}^{t} \max\Big\{ \exp\{3N_{1}(r)\}, \exp\{N_{1}+2N_{2}(r)\}\Big\}\mu_{1}^{-3}(r) dr 
	\geq \frac{1}{4 (\lambda_{11}+\lambda_{12})}, \nonumber\\
	(or) &\int_{0}^{t} \max\Big\{\exp\{3N_{2}(r)\}, \exp\{N_{2}(r)+2N_{1}(r)\} \Big\}\mu_{2}^{-3}(r)dr 
	\geq \frac{1}{4 (\lambda_{21}+\lambda_{22})}\Bigg\}. 
	\end{align} 
That is
$\tau_*$
  corresponds to the first time that  the stochastic processes 
$\mathscr{G}_{1}(t), \mathscr{G}_{1}(t),$ cease to remain strictly positive. Specifically, both processes stay positive for $0\leq t<\tau_*,$
 and at least one of them vanishes at time 
$\tau_*.$

Our first result towards the derivation of the desired lower bound $\tau_*$ is the following:
\begin{theorem}\label{psk1a}
Let $\tau_{*}$ be the stopping time given by \eqref{ST1}). Consider the stochastic processes $\mathscr{G}_{1}(t), \mathscr{G}_{1}(t),$ defined by \eqref{psk1} and \eqref{psk2} respectively  for any $t \in [0,\tau_{*}]$. Then, problem \eqref{rand1a}--\eqref{rand1c}, and thus  \eqref{zeq1}--\eqref{zeq4} as well, admits a (mild) solution $v=(v_{1}, v_{2})^{\top}$  in $[0,\tau_{*}]$ that satisfies 
\begin{align} 
	0<\exp\Big\{{-\frac{k_{11}^{2}}{2} t}\Big\}S_{t}g_{1}(x) \mathscr{G}_{1}(t) &\leq v_{1}(t,x)\leq \exp\Big\{{-\frac{k_{11}^{2}}{2} t}\Big\}S_{t}g_{1}(x)\leq 1,\label{psk3} \\ 0< \exp\Big\{{-\frac{k_{12}^{2}}{2} t}\Big\}S_{t}g_{2}(x) \mathscr{G}_{2}(t) &\leq v_{2}(t,x)\leq \exp\Big\{{-\frac{k_{12}^{2}}{2} t}\Big\}S_{t}g_{2}(x)\leq 1, \label{psk4}
	\end{align}
	for $x \in D$ and $ 0 \leq \tau_{\ast} \leq \tau_q \leq \infty.$ 
\end{theorem}
\begin{proof}
It can be easily seen that 
	\begin{align}
	\frac{d \mathscr{G}_{1}(t)}{dt}&=-4(\lambda_{11}+\lambda_{12}) \max \Big\{\exp\{3N_{1}(t)\}, \exp\{N_{1}(t)+2N_{2}(t)\}\Big\}\mu_{2}^{-3}(t) \mathscr{G}_{2}^{-3}(t), \nonumber\\
    \mathscr{G}_{1}(0)&=1, \nonumber
	\end{align}
	so that 
	\begin{align}
	\mathscr{G}_{1}(t)&=1-(\lambda_{11}+\lambda_{12})\int_{0}^{t} \max \Big\{\exp\{3N_{1}(s)\}, \exp\{N_{1}(s)+2N_{2}(s)\}\Big\} \mu_{1}^{-3}(s) \mathscr{G}_{1}^{-3}(s)ds. \label{psk6}
	\end{align}
	Similarly, we have 
	\begin{eqnarray}
	\mathscr{G}_{2}(t)=1 -(\lambda_{21}+\lambda_{22})\int_{0}^{t} \max\Big\{\exp\{3N_{2}(r)\}, \exp\{N_{2}(r)+2N_{1}(r)\}\Big\} \mu_{2}^{-3}(r) \mathscr{G}_{2}^{-3}(r)dr. \label{psk6a}
	\end{eqnarray}
We now define the operator $\mathcal{R}_{1}$ as
	\begin{align}\label{psk5}
	\mathcal{R}_{1}[V_{1},V_{2}] (t,x)&:=\exp\Big\{{-\frac{k_{11}^{2}}{2} t}\Big\}S_{t}g_{1}(x)-\lambda_{11}\int_{0}^{t} e^{-\frac{k_{11}^{2}}{2}(t-r)+3N_{1}(r)} S_{t-r}\left(V_{1}^{-2}(r,\cdot)\right) (x)dr\nonumber\\
    & \quad -\lambda_{12}\int_{0}^{t} e^{-\frac{k_{11}^{2}}{2}(t-r)+N_{1}(r)+2N_{2}(r)} S_{t-r}\left(V^{-2}_{2}(r,\cdot)\right) (x)dr,
	\end{align}
   where $V_1(\cdot,t), V_2(\cdot,t)\in C_0(D)$ are any non-negative, bounded  functions  satisfying
\begin{equation}\label{psk5a}
0 \leq  \exp\Big\{{-\frac{k_{11}^{2}}{2} t}\Big\}S_{t}g_{1}(x)\mathscr{G}_{1}(t)\leq V_{i}(t,x)\leq \exp\Big\{{-\frac{k_{11}^{2}}{2} t}\Big\}S_{t}g_{1}(x),\ i=1,2,
\end{equation} for $x\in D$ and $0<t<\tau_{*}.$ 

By \eqref{psk5}, since $V_i(\cdot,t)\geq 0, i=1,2$ then $\mathcal{R}_{1} [V_{1},V_{2}](t,x) \leq \exp\Big\{{-\frac{k_{11}^{2}}{2} t}\Big\}S_{t}g_{1}(x).$ Again \eqref{psk5} in conjuction with \eqref{psk5a} reads
	\begin{align}
	\mathcal{R}_{1} [V_{1},V_{2}](t,x)&=\exp\Big\{{-\frac{k_{11}^{2}}{2} t}\Big\}S_{t}g_{1}(x)-\lambda_{11}\int_{0}^{t} e^{-\frac{k_{11}^{2}}{2}(t-r)+3N_{1}(r)} S_{t-r}\left(V_{1}^{-3}(r,\cdot)V_{1}(r,\cdot)  \right)(x)dr\nonumber\\
    & \quad -\lambda_{12}\int_{0}^{t} e^{-\frac{k_{11}^{2}}{2}(t-r)+N_{1}(r)+2N_{2}(r)} S_{t-r}\left(V_{2}^{-3}(r,\cdot)V_{2}(r,\cdot)  \right) (x)dr\nonumber\\
	&\geq \exp\Big\{{-\frac{k_{11}^{2}}{2} t}\Big\}S_{t}g_{1}(x)-\lambda_{11}\int_{0}^{t} e^{-\frac{k_{11}^{2}}{2}(t-r)+3N_{1}(r)} \nonumber\\ &\qquad\qquad\qquad\qquad \times S_{t-r}\left(\left( \exp\Big\{{-\frac{k_{11}^{2}}{2} r}\Big\}S_{r}g_{1}(x)\mathscr{G}_{1}(r)\right)^{-3} V_{1}(r,\cdot)  \right)(x)dr\nonumber\\
    & \quad -\lambda_{12}\int_{0}^{t} e^{-\frac{k_{11}^{2}}{2}(t-r)+N_{1}(r)+2N_{2}(r)}\nonumber\\ &\qquad\qquad\times S_{t-r}\left(\left( \exp\Big\{{-\frac{k_{11}^{2}}{2} r}\Big\}S_{r}g_{1}(x)\mathscr{G}_{1}(r)\right)^{-3} V_{2}(r,\cdot)\right) (x)dr\nonumber\\    
	&\geq \exp\Big\{{-\frac{k_{11}^{2}}{2} t}\Big\}S_{t}g_{1}(x)-\lambda_{11}\int_{0}^{t} e^{-\frac{k_{11}^{2}}{2}(t-r)+3N_{1}(r)} \mu_{1}^{-3}(r) \mathscr{G}_{1}^{-3}(r) \nonumber\\ &\qquad\qquad\qquad\qquad \times S_{t-r}\left( V_{1}(r,\cdot)  \right)(x)dr\nonumber\\
    & \quad -\lambda_{12}\int_{0}^{t} e^{-\frac{k_{11}^{2}}{2}(t-r)+N_{1}(r)+2N_{2}(r)}\mu_{1}^{-3}(r) \mathscr{G}_{1}^{-3}(r) S_{t-r}\left(V_{2}(r,\cdot)\right) (x)dr\nonumber\\ 
	&\geq \exp\Big\{{-\frac{k_{11}^{2}}{2} t}\Big\}S_{t}g_{1}(x)-\lambda_{11}\int_{0}^{t} e^{-\frac{k_{11}^{2}}{2}(t-r)+3N_{1}(r)} \mu_{1}^{-3}(r) \mathscr{G}_{1}^{-3}(r) \nonumber\\ &\qquad\qquad\qquad\qquad \times S_{t-r}\left( \exp\Big\{{-\frac{k_{11}^{2}}{2} r}\Big\}S_{r}g_{1}(x)  \right)dr\nonumber\\
    & \quad -\lambda_{12}\int_{0}^{t} e^{-\frac{k_{11}^{2}}{2}(t-r)+N_{1}(r)+2N_{2}(r)}\mu_{1}^{-3}(r) \mathscr{G}_{1}^{-3}(r)\nonumber\\ &\qquad\qquad\times S_{t-r}\left(\exp\Big\{{-\frac{k_{11}^{2}}{2} r}\Big\}S_{r}g_{1}(x) \right) (x)dr\nonumber\\ 
	&\geq \exp\Big\{{-\frac{k_{11}^{2}}{2} t}\Big\}S_{t}g_{1}(x)-\lambda_{11}\int_{0}^{t} e^{-\frac{k_{11}^{2}}{2}(t-r)+3N_{1}(r)} \mu_{1}^{-3}(r) \mathscr{G}_{1}^{-3}(r) e^{-\frac{k_{11}^{2}}{2} r}S_{t}g_{1}(x)  dr\nonumber\\
    & \quad -\lambda_{12}\int_{0}^{t} e^{-\frac{k_{11}^{2}}{2}(t-r)+N_{1}(r)+2N_{2}(r)}\mu_{1}^{-3}(r) \mathscr{G}_{1}^{-3}(r)e^{-\frac{k_{11}^{2}}{2} r}S_{t}g_{1}(x)dr\nonumber\\ 
	&\geq \exp\Big\{{-\frac{k_{11}^{2}}{2} t}\Big\}S_{t}g_{1}(x)\Big[1 -\lambda_{11}\int_{0}^{t} e^{3N_{1}(r)} \mu_{1}^{-3}(r) \mathscr{G}_{1}^{-3}(r)dr \nonumber\\
    & \quad -\lambda_{12}\int_{0}^{t} e^{N_{1}(r)+2N_{2}(r)}\mu_{1}^{-3}(r) \mathscr{G}_{1}^{-3}(r)dr\Bigg] \nonumber\\
    &\geq\exp\Big\{{-\frac{k_{11}^{2}}{2} t}\Big\}S_{t}g_{1}(x)\Big[1 -(\lambda_{11}+\lambda_{12})\int_{0}^{t} \max \Big\{e^{3N_{1}(r)}, e^{N_{1}(r)+2N_{2}(r)}\Big\} \mu_{1}^{-3}(r) \mathscr{G}_{1}^{-3}(r)dr \Bigg] \nonumber\\
	&=\exp\Big\{{-\frac{k_{11}^{2}}{2} t}\Big\} S_{t}g_{1}(x)\mathscr{G}_{1}(t),\nonumber
	\end{align}
	where the last equality arises from \eqref{psk6}.
    
    Thus, we have
	\begin{equation}\label{psk7}
	    \exp\Big\{{-\frac{k_{11}^{2}}{2} t}\Big\} S_{t}g_{1}(x)\mathscr{G}_{1}(t)\leq \mathcal{R}_{1} [V_{1},V_{2}](t,x)  \leq \exp\Big\{{-\frac{k_{11}^{2}}{2} t}\Big\} S_{t}g_{1}(x).
	\end{equation}
    Next if we define the operator $\mathcal{R}_{2}$ as
    \begin{align*}
	\mathcal{R}_{2}[W_{1},W_{2}] (t,x)&:=\exp\Big\{{-\frac{k_{21}^{2}}{2} t}\Big\}S_{t}g_{2}(x)-\lambda_{21}\int_{0}^{t} e^{-\frac{k_{21}^{2}}{2}(t-r)+3N_{2}(r)} S_{t-r}\left(W_{2}^{-2}(r,\cdot)\right) (x)dr\nonumber\\
    & \quad -\lambda_{22}\int_{0}^{t} e^{-\frac{k_{21}^{2}}{2}(t-r)+N_{2}(r)+2N_{1}(r)} S_{t-r}\left(W^{-2}_{1}(r,\cdot)\right) (x)dr,
	\end{align*}
     for any $W_1(\cdot,t), W_2(\cdot,t)\in C_0(D)$ non-negative, bounded  functions  satisfying
 $$0 \leq  \exp\Big\{{-\frac{k_{21}^{2}}{2} t}\Big\}S_{t}g_{2}(x)\mathscr{G}_{2}(t)\leq |W_{i}(t,x)|\leq \exp\Big\{{-\frac{k_{21}^{2}}{2} t}\Big\}S_{t}g_{2}(x),\ i=1,2.$$ 
	Then  we similarly obtain
	\begin{eqnarray}
	\exp\Big\{{-\frac{k_{21}^{2}}{2} t}\Big\}S_{t}g_{2}(x)\mathscr{G}_{2}(t)\leq \mathcal{R}_{2} [W_{1},W_{2}](t,x)  \leq \exp\Big\{{-\frac{k_{11}^{2}}{2} t}\Big\} S_{t}g_{2}(x). \label{psk8}
	\end{eqnarray}
Now we consider the iteration scheme	
	\begin{eqnarray*}\label{in2}
    &&\thetaa^{(0)}_{1}(t,x):=\exp\Big\{{-\frac{k_{11}^{2}}{2} t}\Big\}S_{t}g_{1}(x), \ \ \ \thetaa^{(0)}_{2}(t,x):=\exp\Big\{{-\frac{k_{12}^{2}}{2} t}\Big\}S_{t}g_{2}(x),\\
	&&\thetaa^{(n)}_{1}(t,x):=\mathcal{R}_{1} [\thetaa^{(n-1)}_{1}, \thetaa^{(n-1)}_{2}](t,x), \ \ \thetaa^{(n)}_{2}(t,x):=\mathcal{R}_{2} [\thetaa^{(n-1)}_{1},\thetaa^{(n-1)}_{2}](t,x), \ \ n\geq1, 
	\end{eqnarray*} 
	for $x\in D,\ 0\leq t<\tau_{*}.$ Our aim is to show that the sequences of functions $\{ \thetaa^{(n)}_{1} \}_{n\in \mathbb{N}} \ \mbox{and} \ \{\thetaa^{(n)}_{2}\}_{n\in \mathbb{N}}$ are decreasing.
	
    Indeed, we have
	\begin{align}
	\thetaa^{(0)}_{1}(t,x)&\geq \exp\Big\{{-\frac{k_{11}^{2}}{2} t}\Big\}-\lambda_{11}\int_{0}^{t} e^{-\frac{k_{11}^{2}}{2}(t-r)+3N_{1}(r)} S_{t-r}\left(\thetaa^{(0)}_{1}(r,x)\right)^{-2} dr\nonumber\\
    &\quad-\lambda_{12}\int_{0}^{t} e^{-\frac{k_{11}^{2}}{2}(t-r)+N_{1}(r)+2N_{2}(r)} S_{t-r}\left(\thetaa^{(0)}_{2}(r,x)\right)^{-2} dr \nonumber\\
	&= \mathcal{R}_{1}[\thetaa^{(0)}_{1},\thetaa^{(0)}_{2}](t,x) = \thetaa^{(1)}_{1}(t,x).\nonumber
	\end{align}
    Similarly
    \begin{align}
	\thetaa^{(0)}_{2}(t,x)&\geq \exp\Big\{{-\frac{k_{12}^{2}}{2} t}\Big\}-\lambda_{21}\int_{0}^{t} e^{-\frac{k_{12}^{2}}{2}(t-r)+3N_{2}(r)} S_{t-r}\left(\thetaa^{(0)}_{2}(r,x)\right)^{-2} dr\nonumber\\
    &\quad-\lambda_{22}\int_{0}^{t} e^{-\frac{k_{12}^{2}}{2}(t-r)+N_{2}(r)+2N_{1}(r)} S_{t-r}\left(\thetaa^{(0)}_{1}(r,x)\right)^{-2} dr \nonumber\\
	&= \mathcal{R}_{2}[\thetaa^{(0)}_{1},\thetaa^{(0)}_{2}](t,x) = \thetaa^{(1)}_{1}(t,x).\nonumber
	\end{align}
	Assuming that $\thetaa^{(n)}_{i}\geq \thetaa^{(n-1)}_{i}, \ i=1,2,$ for some $n \geq 1,$  the monotonicity of $\mathcal{R}_{1}, \mathcal{R}_{2}$ leads to the inequalities
	\begin{eqnarray*}
	&&\thetaa^{(n+1)}_{1}=\mathcal{R}_{1} [ \thetaa^{(n)}_{1},\thetaa^{(n)}_{2}](t,x)\geq \mathcal{R}_{1} [\thetaa^{(n-1)}_{1},\thetaa^{(n-1)}_{2}](t,x)=\thetaa^{(n)}_{1},\\
	&&\thetaa^{(n+1)}_{2}=\mathcal{R}_{2}  [\thetaa^{(n)}_{1},\thetaa^{(n)}_{2}](t,x)\geq \mathcal{R}_{2} [\thetaa^{(n-1)}_{1},\thetaa^{(n-1)}_{2}](t,x)=\thetaa^{(n)}_{2},
	\end{eqnarray*}
    which by induction implies the monotonicity of the sequences $\{ \thetaa^{(n)}_{1} \}_{n\in \mathbb{N}} \ \mbox{and} \ \{\thetaa^{(n)}_{2}\}_{n\in \mathbb{N}}.$
    
	Therefore the limits 
	\begin{eqnarray}
	\tilde{v}_{1}(t,x)=\lim_{n\rightarrow \infty}\thetaa^{(n)}_{1}(t,x), \ \ \tilde{v}_{2}(t,x)=\lim_{n\rightarrow \infty}\thetaa^{(n)}_{2}(t,x), \nonumber
	\end{eqnarray}
	exist for $x\in D$ and $0 \leq t< \tau_{*}.$ Then by the monotone convergence theorem for decreasing functions, we obtain 
	\begin{eqnarray}
	\tilde{v}_{1}(t,x)=\mathcal{R}_{1} [\tilde{v}_{1},\tilde{v}_{2}](t,x), \  \tilde{v}_{2}(t,x)=\mathcal{R}_{2}[ \tilde{v}_{2},\tilde{v}_{1}](t,x), \ \ x\in D, \ \ 0 \leq t<\tau_{*}.\nonumber
	\end{eqnarray}
	Since system \eqref{rand1a}--\eqref{rand1c}  has a unique solution $v=(v_{1},v_{2})^{\top}$ then $\tilde{v}=(\tilde{v}_{1},\tilde{v}_{2})^{\top}$ coincides with that solution, so we obtain 	the desired estimates \eqref{psk3} and \eqref{psk4}, that is
		\begin{align} 
	0<\exp\Big\{{-\frac{k_{11}^{2}}{2} t}\Big\}S_{t}g_{1}(x) \mathscr{G}_{1}(t) \leq v_{1}(t,x)&=\mathcal{R}_{1} [v_{1},v_{2}] (t,x)\leq \exp\Big\{{-\frac{k_{11}^{2}}{2} t}\Big\}S_{t}g_{1}(x)\leq 1\ \text{ and } \nonumber \\ 0< \exp\Big\{{-\frac{k_{12}^{2}}{2} t}\Big\}S_{t}g_{2}(x) \mathscr{G}_{2}(t) \leq v_{2}(t,x)&=\mathcal{R}_{2} [v_{2},v_{1}](t,x)\leq \exp\Big\{{-\frac{k_{12}^{2}}{2} t}\Big\}S_{t}g_{2}(x)\leq 1, \nonumber
	\end{align}
	for all \( x \in D \) and \( 0 \leq \tau_{\ast} \leq \tau_q \leq \infty \). Note that the inequality
\[
0 < S_{t}g_{i}(x) \leq 1,
\]
holds for \( 0 < g_i(x) \leq 1 \), where \( i = 1, 2 \), and \( S_t g_i(x) \) denotes the solution to the corresponding linear problem. This establishes the desired result and thus completes the proof.

\end{proof}
\begin{Rem}\label{sk1}
Observe that, based on the estimates \eqref{psk3} and \eqref{psk4}, we can conclude that the quenching time \( \tau_q \) for the solution \( v=(v_1,v_2)^{\top} \) of problem \eqref{rand1a}--\eqref{rand1c}, as well as for the solution \( z=(z_1,z_2)^{\top} \) of the system \eqref{zeq1} -- \eqref{zeq4}, is bounded below by the random variable \( \tau_{\ast} \) defined in \eqref{ST1}. That is, $
\tau_{\ast} \leq \tau_q.$

 \end{Rem}
Next, by applying Theorem \ref{psk1a}, we establish a condition under which problem \eqref{rand1a}--\eqref{rand1c}, as well as  system \eqref{zeq1}--\eqref{zeq4}, admits a global-in-time solution almost surely.

\begin{cor}\label{sk3}
Consider initial data $0<v_i(0,x)=g_i(x)\leq 1$  satisfying the condition
\begin{eqnarray}\label{ik22}
\int_0^{\infty} \max\{e^{3 N_i(r)}, e^{ N_i(r)+2N_j(r)} \}\mu_i^{-3}(r)\, dr<\frac{1}{4 (\lambda_{i1}+\lambda_{i2}) },\quad\mbox{for}\quad i=1,2,\quad\mbox{and}\quad j\in\{1,2\}\setminus \{i\}, \quad
\end{eqnarray}
recalling that $\mu_{i}(t):= \displaystyle  \exp\Big\{{-\frac{k_{i1}^{2}}{2} t}\Big\} \inf_{x \in D} S_{t} g_{i}(x)>0, \,i=1,2.$

Then problem \eqref{rand1a}--\eqref{rand1c}, and thus  \eqref{zeq1} -- \eqref{zeq4} as well, admits a  global-in-time solution  with probability $1.$ Furthermore, the solution \( v=(v_1,v_2)^{\top} \) of \eqref{rand1a}--\eqref{rand1c} fulfills the following estimate 
\begin{eqnarray}\label{ik21b}
0<\exp\Big\{{-\frac{k_{i1}^{2}}{2} t}\Big\}S_{t}g_{i}(x) \mathscr{G}_{i}(t) \leq v_i(t,x)\leq  1,\;x\in D, \end{eqnarray}
for any $t\geq 0.$
\end{cor}\label{thm3}
\begin{proof}
From \eqref{ST1}, and in light of \eqref{ik22}, we obtain that \( \tau_{\ast} = \infty \). Consequently, the desired estimate \eqref{ik21b} holds by virtue of \eqref{psk3} and \eqref{psk4} for all \( t \geq 0 \), and we therefore conclude that \( \tau_q = \infty \).
\end{proof}
In the following, we derive a sufficient criterion that ensures the validity of~\eqref{ik22}. This criterion is formulated in terms of the principal eigenpair \((\chi, \psi)\) of the eigenvalue problem~\eqref{a3}--\eqref{a3a}, with the eigenfunction \(\psi\) normalized so that
\(
\int_{D} \psi(x)\,dx = 1.
\)

 We consider initial data \begin{align}\label{cond1}
0<L_{i}S_{\xi_i}\psi(x) \leq z_{i}(0,x)= g_{i}(x) \leq 1,\ \ x \in D,\ \ i=1,2,
\end{align}
where $\xi_i \geq 1, \, i=1,2$ are fixed and for some positive constants $L_{i}, i=1,2$   to be specified in the sequel.
Set $\psi_m:=\inf_{x\in D} \psi>0$, then \eqref{cond1} in conjunction with \eqref{SK1} for $i=1,2$ yields
\begin{align}\label{ik23}
[S_{t}g_i](x) &\geq L_i \left[S_{t+\xi_i} \psi\right](x)\nonumber\\
&= L_i \left(e^{-(\chi+\xi_i) t} \psi(x)\right)\nonumber\\
&\geq L_i \psi_m e^{-(\chi+\xi_i) t},\quad\mbox{for any}\quad x\in D,\; t\geq 0,
\end{align}
where the lower  bound in \eqref{ik23} is independent of the spatial variable $x.$ 

Since the function $(x,t)\mapsto [S_{t}\psi](x)$ is uniformly bounded in $x,$  then \eqref{46sk} thanks to \eqref{ik23} reads
\begin{eqnarray*}
\mu_i(t)\geq L_i \psi_m  \exp\Big\{{-\left(\frac{k_{i1}^{2}}{2}+\chi+\xi_i\right) t}\Big\}\quad\mbox{for any}\quad  t\geq 0,
\end{eqnarray*}
and thus condition \eqref{ik22} is satisfied provided that
\begin{eqnarray*}
\left(L_i \psi_m \right)^{-3}\int_0^{\infty} \max\{e^{3 N_i(r)}, e^{ N_i(r)+2N_j(r)} \}e^{3(\frac{k_{i1}}{2}+\chi+\xi_i)r}\, dr<\frac{1}{4 (\lambda_{i1}+\lambda_{i2})},
\end{eqnarray*}
or equivalently
\begin{eqnarray}\label{ik24}
\int_0^{\infty} \max\{e^{3 N_i(r)}, e^{ N_i(r)+2N_j(r)} \}e^{3(\frac{k_{i1}}{2}+\chi+\xi_i)r}\, dr<\widetilde{L}_i,\qquad
\end{eqnarray}
for $\displaystyle\widetilde{L}_i:=\frac{\left(L_i \psi_m \right)^{3}}{4 (\lambda_{i1}+\lambda_{i2}) }, \, i=1.2.$

This leads us to a more refined and specific result concerning global existence.
\begin{proposition}\label{ps1}
Under conditions  \eqref{cond1}, \eqref{ik24} for some  $W_i>0,\, \xi_i, i=1,2,$ then, problem \eqref{rand1a}--\eqref{rand1c}, and thus  \eqref{zeq1} -- \eqref{zeq4} as well,, has  a global in time solution  with probability $1$ (almost surely).
\end{proposition}
\begin{proof}
Note that conditions \eqref{cond1}, \eqref{ik24} imply the vailidity of \eqref{ik22}, and thus the result follows from Corollary \ref{sk3}.
\end{proof}
\begin{Rem}
Under the special case
\begin{equation}\label{eq1}
		\left\{
		\begin{aligned}
			k_{21}+2k_{11}=3k_{11}:=\rho_{11},\ k_{22}+2k_{12}=3k_{12}:=\rho_{12}, \\
			k_{11}+2k_{21}=3k_{21}:=\rho_{21},\ k_{12}+2k_{22}=3k_{22}:=\rho_{22}, 
		\end{aligned}
		\right.
	\end{equation}
    we deduce
\begin{equation}\label{eq1a}
		\left\{
		\begin{aligned}
			3 N_1(t)=\rho_{11} W(t)+\rho_{12} B^H(t),\ N_1(t)+2 N_2(t)=\rho_{21} W(t)+\rho_{22} B^H(t), \\
            3N_2(t)=\rho_{21} W(t)+\rho_{22} B^H(t),\ N_2(t)+2 N_1(t)=\rho_{11} W(t)+\rho_{12} B^H(t). 
		\end{aligned}
		\right.
	\end{equation}
Then for general initial data, by virture of \eqref{ST1} and \eqref{eq1a}, we deduce a lower bound of the quenching time as
\begin{align*}
	\tau_{\ast} = \inf \Bigg\{ t\geq 0 : &\int_{0}^{t} \max\Big\{ \exp\{\rho_{11} W(r)+\rho_{12} B^H(r)\}, \exp\{\rho_{21} W(r)+\rho_{22} B^H(r)\}\Big\}\mu_{1}^{-3}(r) dr \geq \frac{1}{4 (\lambda_{11}+\lambda_{12})}, \nonumber\\
	(or) &\int_{0}^{t} \max\Big\{\exp\{\rho_{21} W(r)+\rho_{22} B^H(r)\}, \exp\{\rho_{11} W(r)+\rho_{12} B^H(r)\} \Big\}\mu_{2}^{-3}(r)dr \geq \frac{1}{4 (\lambda_{21}+\lambda_{22})}\Bigg\}, 
	\end{align*}
while for initial data satisfying \eqref{cond1}, taking also into account \eqref{ik24}, we obtain
\begin{align}\label{ST1a}
	\tau_{\ast} = \inf \Bigg\{ t\geq 0 : &\int_{0}^{t} \max\Big\{ \exp\{\rho_{11} W(r)+\rho_{12} B^H(r)\}, \exp\{\rho_{21} W(r)+\rho_{22} B^H(r)\}e^{3(\frac{k_{11}}{2}+\chi+\xi_1)r}\,dr 
	\geq \widetilde{L}_1, \nonumber\\
	(or) &\int_{0}^{t} \max\Big\{\exp\{\rho_{21} W(r)+\rho_{22} B^H(r)\}, \exp\{\rho_{11} W(r)+\rho_{12} B^H(r)\} \Big\}e^{3(\frac{k_{21}}{2}+\chi+\xi_2)r}\,dr 
	\geq \widetilde{L}_2. 
\end{align}
\end{Rem}
	\subsection{Upper bound for the quenching time}\label{sec4}
	Next we proceed to obtain an upper bound $\tau^*$ for the quenching time $\tau_q$ for set of parameters satisfying \eqref{eq1}.
	
	We first recall that for any test function $\varphi_{i}\in C^2(D),\ i=1,2$, satisfying the boundary condition \eqref{a3a}  the weak formulation of system \eqref{rand1a}--\eqref{rand1c} reduces to
	\begin{align} \label{a5}
	v_{i}(t,\phi_i) &= v_{i}(0,\phi_i)+\int_{0}^{t}v_{i}\left(s,\Delta \phi_i\right)ds-\lambda_{i1}\int_{0}^{t}v^{-2}_{i}\left(s,\phi_i\right) \exp\{{3N_{i}(s)}\}ds\nonumber\\
	&\quad-\lambda_{i2}\int_{0}^{t}v^{-2}_{j}\left(s,\phi_i\right) \exp\{{N_{i}(s)+2N_{j}(s)}\}ds-\frac{k_{i1}^{2}}{2} \int_{0}^{t}v_{i}(s, \phi_i)ds,	
	\end{align}		
    for $i=1,2, \ j\in \left\{ 1,2 \right\} / \left\{i\right\}.$ 
    
    Now, if we choose as both test functions the first eigenfunction $\psi$ of the Laplacian operator $-\Delta_R$ satisfying the eigenvalue problem~\eqref{a3}--\eqref{a3a} then  \begin{equation*}
    \begin{aligned}
    v_{i}(s,\Delta \phi_i)&= v_{i}(s,\Delta \psi)\\&=\int_{D}v_{i}(s,x)\Delta \psi(x)dx\\
    &=-\int_{D}  v_{i}(s,x) \chi\psi(x)dx\\&=-\chi v_{i}(s,\psi),\ i=1,2.
    \end{aligned}
    \end{equation*}
Therefore the weak formulation \eqref{a5} reads
\begin{align} \label{a5a}
	v_{i}(t,\psi) &= v_{i}(0,\phi_i)-\left(\chi+\frac{k^2_{i1}}{2}\right)\int_{0}^{t}v_{i}\left(s, \psi\right)ds-\lambda_{i1}\int_{0}^{t}v^{-2}_{i}\left(s,\phi_i\right) \exp\{{3N_{i}(s)}\}ds\nonumber\\
	&\quad-\lambda_{i2}\int_{0}^{t}v^{-2}_{j}\left(s,\phi_i\right) \exp\{{N_{i}(s)+2N_{j}(s)}\}ds,	
	\end{align}	
for $i=1,2, \ j\in \left\{ 1,2 \right\} / \left\{i\right\}.$ 

The following theorem provides an upper bound for the quenching time \( \tau_q \) of the solution \( v = (v_{1}, v_{2})^{\top} \) to system \eqref{rand1a}--\eqref{rand1c}, and consequently for the system \eqref{zeq1}--\eqref{zeq4} as well.

    \begin{theorem} \label{thm4.1}
Consider the random (stopping) time 
\begin{align}\label{ST2}
    	\tau^{\ast}:=\inf \Bigg\{ t \geq 0 : &\int_{0}^{t} \min\left\{  \exp\{\rho_{11}W(s)+\rho_{12}B^{H}(s)\}, \exp\{\rho_{21}W(s)+\rho_{22}B^{H}(s)\}\right\}\nonumber\\
        &\qquad\qquad \times \exp\{3(\chi+k^2)s\} ds \geq \frac{E^{3}(0)}{ 12\tilde{\lambda}} \Bigg\}, 
    	\end{align}
where $\tilde{\lambda}:=\min\{\lambda_{11}+\lambda_{22}, \lambda_{12}+\lambda_{21}\}$, $k^{2}:=\displaystyle\min\left\lbrace  \frac{k_{11}^{2}}{2}, \frac{k_{21}^{2}}{2} \right\rbrace$, $E(0):=\displaystyle\int_{D}[g_{1}(x)+g_{2}(x)]\psi(x)dx,$ and $\rho_{11}, \rho_{12}, \rho_{21}$ and $\rho_{22}$ are given in \eqref{eq1}. 

Then, on the event $\{\tau^*<\infty\}$ the solution \( v = (v_{1}, v_{2})^{\top} \) of problem \eqref{rand1a}--\eqref{rand1c}, and thus the solution $z = (z_{1}, z_{2})^{\top}$ of  \eqref{zeq1}--\eqref{zeq4}, quenches in finite $\tau_q \leq \tau^*,\;\; \mathbb{P}-a.s.$
\end{theorem}	    
      \begin{proof} 
Note that by \eqref{a5a} and for any $\ve>0$ we have
\begin{align*}
\frac{v_{i}(t+\ve,\psi)- v_i(t, \psi) }{\ve}=&-\left(\chi+\frac{k^2_{i1}}{2}\right)\frac{1}{\ve}\int_t^{t+\ve} v_i(s,\psi)\, ds-\frac{\lambda_{i1}}{\ve}\int_{t}^{t+\ve}v^{-2}_{i}\left(s,\psi\right) \exp\{{3N_{i}(s)}\}ds\\&-\frac{\lambda_{i2}}{\ve}\int_{t}^{t+\ve}v^{-2}_{j}\left(s,\psi\right) \exp\{{N_{i}(s)+2N_{j}(s)}\}ds,
\end{align*}
for $i=1,2, \ j\in \left\{ 1,2 \right\} / \left\{i\right\}.$ 

Letting now $\ve\to 0$
	\begin{align*}
	\frac{d v_{i}(t,\psi)}{dt} &= -\left(\chi+\frac{k_{i1}^{2}}{2} \right) v_{i}\left(t,\psi\right) -\lambda_{i1}v^{-2}_{i}\left(t,\psi\right) \exp\{{3N_{i}(t)}\} \nonumber\\
	&\quad-\lambda_{i2}v^{-2}_{j}\left(t,\psi\right) \exp\{{N_{i}(t)+2N_{j}(t)}\},
	\end{align*}	
for $i=1,2, \ j\in \left\{ 1,2 \right\} / \left\{i\right\}.$ 

	Utilizing the Jensen's inequality,  recalling also that $\int_D \psi(x)\,dx=1,$ we obtain 
	\begin{align}
	v_{i}^{-2}(t, \psi):=\int_{D} v^{-2}_{i}(t,x)\psi(x)dx \geq \left[ \int_{D} v_{i}(t,x)\psi(x)dx\right]^{-2}  =\frac{1}{[v_{i}(t, \psi)]^{2}},\ i=1,2, \nonumber
	\end{align}
	hence, since $k^{2}=\displaystyle\min\left\lbrace  \frac{k_{11}^{2}}{2}, \frac{k_{21}^{2}}{2} \right\rbrace,$ we derive
	\begin{align*}
	\frac{d v_{i}(t,\psi)}{dt} &\leq -\left( \chi+\frac{k_{i1}^{2}}{2} \right) v_{i}\left(t,\psi\right) -\lambda_{i1} \exp\{{3N_{i}(t)}\}[v_{i}(t, \psi)]^{-2} \nonumber\\
	&\quad-\lambda_{i2} \exp\{{N_{i}(t)+2N_{j}(t)}\}[v_{j}(t, \psi)]^{-2}\nonumber\\
	&\leq -\left(\chi+k^{2} \right) v_{i}\left(t,\psi\right) -\lambda_{i1} \exp\{{3N_{i}(t)}\}[v_{i}(t, \psi)]^{-2}\nonumber\\
	&\quad-\lambda_{i2} \exp\{{N_{i}(t)+2N_{j}(t)}\}[v_{j}(t, \psi)]^{-2}, 
	\end{align*}
	for $i=1,2, \ j\in \left\{ 1,2 \right\} / \left\{i\right\}.$ 
    
    In this manner, and by utilizing \eqref{cnt} and \eqref{eq1}, we deduce that
\[
v_{i}(t, \psi) \leq h_{i}(t), \quad \text{for } i = 1, 2,
\]
where

	\begin{equation} \label{ab2}
	\begin{aligned}
	\frac{dh_{1}(t)}{dt}&=-\left( \chi+k^{2}\right) h_{1}(t)-\lambda_{11} \exp\{{\rho_{11}W(t)+\rho_{12}B^{H}(t)}\}h_{1}^{-2}(t)  \nonumber\\
	& \quad -\lambda_{12} \exp\{{\rho_{21}W(t)+\rho_{22}B^{H}(t)}\}h_{2}^{-2}(t),\\ 
	\frac{dh_{2}(t)}{dt}&=-\left( \chi+k^{2}\right) h_{2}(t)-\lambda_{21} \exp\{{\rho_{21}W(t)+\rho_{22}B^{H}(t)}\}h_{2}^{-2}(t),\\&\quad -\lambda_{22} \exp\{{\rho_{11}W(t)+\rho_{12}B^{H}(t)}\}h_{1}^{-2}(t)\\
	h_{i}(0)&=v_{i}(0,\psi), \ i=1,2. 
	\end{aligned}
	\end{equation}
	 Let  us define $E(t):=h_{1}(t)+h_{2}(t)\geq 0, \ t\geq 0,$ then $[h_{1}^{-2}(t)+h_{2}^{-2}(t)]\geq 4 E^{-2}(t)$ and thus $E(t)$ satisfies
	\begin{align*}
	\frac{dE(t)}{dt} &= -\left( \chi+k^{2}\right)E(t)-(\lambda_{11}+\lambda_{22}) \exp\{{\rho_{11}W(t)+\rho_{12}B^{H}(t)}\}h_{1}^{-2}(t)\nonumber\\
	&\quad-(\lambda_{12}+\lambda_{21}) \exp\{{\rho_{21}W(t)+\rho_{22}B^{H}(t)}\}h_{2}^{-2}(t) \nonumber\\
	&\leq -\left( \chi+k^{2}\right)E(t)\nonumber\\
    &\quad-\tilde{\lambda}\min\left\{ \exp\{{\rho_{11}W(t)+\rho_{12}B^{H}(t)}\}, \exp\{{\rho_{21}W(t)+\rho_{22}B^{H}(t)}\}\right\}[h_{1}^{-2}(t)+h_{2}^{-2}(t)]\nonumber\\
	&\leq-\left( \chi+k^{2}\right)E(t)- 4\tilde{\lambda}\min\left\{ \exp\{{\rho_{11}W(t)+\rho_{12}B^{H}(t)}\}, \exp\{{\rho_{21}W(t)+\rho_{22}B^{H}(t)}\}\right\}E^{-2}(t),\nonumber
	\end{align*}
	for  $\tilde{\lambda}=\min\{\lambda_{11}+\lambda_{22},\lambda_{12}+\lambda_{21}\}$. Here, by a comparison argument (see \cite[Theorem 1.3]{teschl}), we have $E(t) \leq I(t)$ for all $t \geq 0$, where $I(t)$ solves the  Bernoulli differential equation
	\begin{align}
	\frac{dI(t)}{dt} &= -\left( \chi+k^{2}\right)I(t)- 4\tilde{\lambda}\min\left\{ \exp\{{\rho_{11}W(t)+\rho_{12}B^{H}(t)}\}, \exp\{{\rho_{21}W(t)+\rho_{22}B^{H}(t)}\}\right\}I^{-2}(t), \nonumber\\
		I(0)&=E(0), \nonumber 
	\end{align}
	hence
	\begin{align}\label{kik1}
	I(t) &= \exp\{-(\chi+k^2)t\}\Bigg[I^{3}(0) \nonumber\\
    &- 12\tilde{\lambda}\int_{0}^{t} \min\left\{ \exp\{{\rho_{11}W(s)+\rho_{12}B^{H}(s)}\}, \exp\{{\rho_{21}W(s)+\rho_{22}B^{H}(s)}\}\right\} \exp\{3(\chi+k^2)s\} ds \Bigg]^{\frac{1}{3}},  
	\end{align}
	for $0 \leq t \leq \tau^{\ast}.$ This, yields that the stopping time $\tau^{\ast}$ is given by \eqref{ST2}.
	Thus, $I(\cdot)$ hits $0,$ that is it quenches in finite-time on the event $\{\omega \in \Omega \ ;\tau^{\ast}(\omega)< \infty\}.$ Since, $I(\cdot)\geq E(\cdot) =h_{1}(\cdot)+h_{2}(\cdot) \geq v_{1}(\cdot, \psi)+v_{2}(\cdot, \psi),$ $\tau^{\ast}$ is an upper bound for the quenching (stopping) time $\tau_q$ of $v=(v_1,v_2)^{\top},$
	which completes the proof.
\end{proof}
\begin{Rem}\label{mmnk21a}
For the special case when  $\rho_{11}=\rho_{21}:=\rho_{1}$ and $\rho_{12}=\rho_{22}:=\rho_{2}$ the upper bound of the quenching time given by \eqref{ST2} reduces to
\begin{align}\label{ST2b}
    	\tau^{\ast}:=\inf \Bigg\{ t \geq 0 : &\int_{0}^{t}   \exp\{3(\chi+k^2)s+\rho_{1}W(s)+\rho_{2}B^{H}(s)\}ds \geq \frac{E^{3}(0)}{12\tilde{\lambda}} \Bigg\}.
    	\end{align}
\end{Rem}
Without the condition \eqref{eq1}, an upper bound of the  quenching time of the system \eqref{zeq1}--\eqref{zeq4} is given by the following corollary.
\begin{cor}\label{ps2}
	Consider the random (stopping) time 
	\begin{align*}\label{ST2b}
		\tau^{\ast \ast }:=\inf \Bigg\{ t \geq 0 : &\int_{0}^{t} \min \Bigg\{  \exp\{3k_{11}W(s)+3k_{12}B^{H}(s)\}, \exp\{(k_{11}+2k_{21})W(s)+(k_{12}+2k_{22})B^{H}(s)\},\nonumber\\
		& \qquad \qquad \exp\{3k_{21}W(s)+3k_{22}B^{H}(s)\}, \exp\{(k_{21}+2k_{11})W(s)+(k_{22}+2k_{12})B^{H}(s)\}\Bigg\} \nonumber\\
		&\qquad\qquad\quad \times \exp\{3(\chi+k^2)s\} ds \geq \frac{E^{3}(0)}{ 12\tilde{\lambda}} \Bigg\}, 
	\end{align*}
	where  the quantities $k^{2}$, $E(0)$, and $\tilde{\lambda}$ are defined as in Theorem~\ref{thm4.1}.
	
	Then, on the event $\{\tau^{\ast \ast }<\infty\}$ the solution \( v = (v_{1}, v_{2})^{\top} \) of problem \eqref{rand1a}--\eqref{rand1c}, and thus the solution $z = (z_{1}, z_{2})^{\top}$ of  \eqref{zeq1}--\eqref{zeq2}, quenches in finite $\tau_q,$ where $\tau_q \leq \tau^{**},\;\; \mathbb{P}-a.s.$
\end{cor}




\begin{Rem}\label{mmnk2}
	For the quenching time \(\tau_q\) of the solution  \(v = (v_{1}, v_{2})^{\top}\) of problem ~\eqref{rand1a}--\eqref{rand1c}, or equivalently  of \(z = (z_{1}, z_{2})^{\top}\) solving system~\eqref{zeq1}--\eqref{zeq4}, the random variables \(\tau_{\ast}\) and \(\tau^{\ast}\), defined by~\eqref{ST1} and~\eqref{ST2}, respectively, can be employed to construct an estimation interval \([\tau_{\ast}, \tau^{\ast}]\) for the quenching time \(\tau_q\) under condition \eqref{eq1}. 
			
	Indeed, if we choose initial data in the form
	\[g_{i}(x) = L_i e^{-\chi \xi_i} \psi(x), \quad i = 1, 2,\]
	for some constants \(\xi_i \geq 1\) and \(L_i > 0\), \(i = 1, 2\), then condition~\eqref{cond1} is satisfied. In this case, the bounds \(\tau_{\ast}\) and \(\tau^{\ast}\) can be explicitly expressed in terms of exponential functions involving a combination of Brownian and fractional Brownian motions and the first Robin eigenfunction $\psi(x)$ of $-\Delta$ solving problem \eqref{a3}--\eqref{a3a}. 
			
	Specifically, under this choice, we have:	
	\begin{align}\label{A1}
	\mu_{i}(t)&\geq L_i e^{-(\chi+k^{2})t} e^{-\xi_i \chi} \inf_{x \in D} \psi(x)\nonumber\\
	&\geq L_m e^{-(\chi+k^{2})t} e^{-\xi_M \chi} \inf_{x \in D} \psi(x),
	\end{align}	
	for $i=1,2$ where $L_m:=\min\{L_1,L_2\}$ and $\xi_M:=\max\{\xi_1,\xi_2\}.$

    Also
	\begin{align}\label{A2}
	E(0)\geq 2L_me^{-\chi \xi_M }\int_{D}\psi^{2}(x)dx.
    \end{align}
 Then, \eqref{ST1} in conjunction with \eqref{eq1a} and \eqref{A1} gives
    \begin{align*}
    \tau_{\ast} = \inf \Bigg\{ t\geq 0 : &\int_{0}^{t} \max\Big\{ \exp\{\rho_{11} W(r)+\rho_{12} B^H(r)\}, \exp\{\rho_{21} W(r)+\rho_{22} B^H(r)\}\Big\}e^{3(\chi+k^{2})r} dr \nonumber\\
    & \hspace{1.7 in} \geq \frac{L_{m}^{3}e^{-3\chi \xi_M} (\inf_{x \in D}\psi(x))^{3}}{4 (\lambda_{11}+\lambda_{12})}, \nonumber\\
    (or) &\int_{0}^{t} \max\Big\{\exp\{\rho_{21} W(r)+\rho_{22} B^H(r)\}, \exp\{\rho_{11} W(r)+\rho_{12} B^H(r)\} \Big\}e^{3(\chi+k^{2})r}dr \nonumber\\
    & \hspace{1.7 in} \geq \frac{L_{m}^{3}e^{-3\chi \xi_M} (\inf_{x \in D}\psi(x))^{3}}{4 (\lambda_{21}+\lambda_{22})}\Bigg\}.
    \end{align*}
    Therefore,
    \begin{align}
	\tau_{\ast} = \inf \Bigg\{ t\geq 0 : &\int_{0}^{t} \max \Big\{\exp\{\rho_{11} W(r)+\rho_{12} B^{H}(r)\}, \exp\{\rho_{21} W(r)+\rho_{22} B^{H}(r)\} \Big\}\nonumber\\
	&\qquad\qquad\times \exp\{3(\chi+k^{2})r\}dr 
	\geq \frac{L_{m}^{3}e^{-3\chi \xi_M} (\inf_{x \in D}\psi(x))^{3}}{4 \tilde{\lambda}}	\Bigg\}. \nonumber
	\end{align}
    recalling that $\tilde{\lambda}=\min\{\lambda_{11}+\lambda_{22}, \lambda_{12}+\lambda_{21}\}.$
    
	Also, \eqref{ST2} by virtue of \eqref{A2}, implies
	\begin{align}
	\tau^{\ast}=\inf \Bigg\{ t \geq 0 : \int_{0}^{t} &\min \Big\{  \exp\{\rho_{11}W(r)+\rho_{12}B^{H}(r)\}, \exp\{\rho_{21}W(s)+\rho_{22}B^{H}(s)\}\Big\}\nonumber\\
	&\qquad \times \exp\{3(\chi+k^2)r\} dr \geq \frac{8L_m^{3}e^{-3\chi \xi_M}\left( \int_{D} \psi^{2}(x)dx \right)^{3}}{12 \tilde{\lambda}} \Bigg\}. \nonumber 
	\end{align} 
	
	Note that  since
	\begin{align*}
	 &\int_{0}^{t}\max \Big\{\exp\{\rho_{11} W(r)+\rho_{12} B^{H}(r)\}, \exp\{\rho_{21} W(r)+\rho_{22} B^{H}(r)\} \Big\} dr \nonumber\\ &\hspace{0.6 in}\geq \int_{0}^{t}\min \Big\{\exp\{\rho_{11} W(r)+\rho_{12} B^{H}(r)\}, \exp\{\rho_{21} W(r)+\rho_{22} B^{H}(r)\} \Big\} dr, \nonumber
	\end{align*}
    then \begin{align}
	\tau^{\ast}=\inf \Bigg\{ t \geq 0 : \int_{0}^{t} &\max \Big\{  \exp\{\rho_{11}W(r)+\rho_{12}B^{H}(r)\}, \exp\{\rho_{21}W(s)+\rho_{22}B^{H}(s)\}\Big\}\nonumber\\
	&\qquad \times \exp\{3(\chi+k^2)r\} dr \geq \frac{2L_m^{3}e^{-3\chi \xi_M}\left( \int_{D} \psi^{2}(x)dx \right)^{3}}{3 \tilde{\lambda}} \Bigg\}. \nonumber 
	\end{align} 
	Consequently, 
    $\tau_{\ast} \leq \tau^{\ast},\;\; \mathbb{P}-a.s.$ provided that 
	\begin{align*}
	\frac{3 (\inf_{x \in D}\psi(x))^{3}}{8 } \leq (\inf_{x \in D}\psi(x))^{3} \leq \left(\int_{D}\psi^{2}(x) dx\right)^{3}.
	\end{align*} 
    The latter inequality is readily seen to be always true since  $\int_D \psi(x)\,dx=1.$ 
	\end{Rem}

\subsection{Estimates of the quenching rate}

The following result provides both lower and upper estimates for the quenching rate of solutions \( v = (v_1, v_2)^{\top} \) to system~\eqref{rand1a}--\eqref{rand1c}.

\begin{theorem}\label{Athm6.6}
Assume that condition~\eqref{eq1} holds, and let the initial data be of the form \( g_i(x) = C_i \psi(x) \)  with constants \( C_i > 0 \) for \( i = 1,2 \). Then the solution to system~\eqref{rand1a}--\eqref{rand1c}. satisfies the following lower bound:
\begin{align}\label{ak71}
0 < C_0 \psi_m e^{-(\chi + k^2)t} \min\{\mathscr{G}_1(t), \mathscr{G}_2(t)\} 
\leq \min\left\{ \min_{x \in \bar{D}} |v_1(t,x)|, \min_{x \in \bar{D}} |v_2(t,x)| \right\}, \quad0 \leq \tau_* \leq t < \infty,
\end{align}
where \( C_0 := \min\{C_1, C_2\} \), \( \psi_m := \min_{x \in \bar{D}} \psi(x) \), and \( \mathscr{G}_i(t) \) are defined in~\eqref{psk6}--\eqref{psk6a}.

Moreover, the following upper bound holds:
\begin{align}\label{ak71b}
\min\left\{ \min_{x \in \bar{D}} |v_1(t,x)|, \min_{x \in \bar{D}} |v_2(t,x)| \right\}
\leq \frac{1}{2} e^{-(\chi + k^2)t} \mathcal{I}(t), \quad 0 \leq t \leq \tau^* < \infty,
\end{align}
where
\begin{align}\label{kik2}
\mathcal{I}(t) := \left[ (C_1 + C_2)^3 \left( \int_D \psi^2(x)\,dx \right)^3 
- 12\tilde{\lambda} \int_0^t \max_{i=1,2} \left\{ e^{\rho_{i1}W(s) + \rho_{i2} B^H(s)} \right\} e^{3(\chi + k^2)s} \, ds \right]^{1/3}.
\end{align}
\end{theorem}

\begin{proof}
Substituting \( g_i(x) = C_i \psi(x) \) into~\eqref{psk3}--\eqref{psk4} and applying~\eqref{SK1}, we obtain
\begin{align}\label{ak71a}
0 < C_i e^{-(\chi + k^2)t} \psi(x) \mathscr{G}_i(t) \leq v_i(t,x), \quad x \in D,\ t \geq \tau_\ast,\ i = 1,2,
\end{align}
with 
\[
k^2 := \min\left\{ \frac{k_{11}^2}{2}, \frac{k_{21}^2}{2} \right\},
\quad 
\mathscr{G}_i(t) := \left[ 1 - 4(\lambda_{i1} + \lambda_{i2}) \int_0^t \max_{j=1,2} \left\{ e^{\rho_{j1}W(s) + \rho_{j2} B^H(s)} \right\} \mu_i^{-3}(s) \, ds \right]^{1/4},
\]
and \( \mu_i(t) := \inf_{x \in D} e^{-\frac{k_{i1}^2}{2} t} S_t g_i(x) > 0 \). This yields the lower bound~\eqref{ak71}.

To establish the upper bound, we use the estimate from Theorem~\ref{thm4.1}:
\begin{align}\label{kik3}
v_1(t,\psi) + v_2(t,\psi) \leq e^{-(\chi + k^2)t} \mathcal{I}(t),
\end{align}
where \( \mathcal{I}(t) \) is given by~\eqref{kik2}, based on the special choice of initial data. On the other hand, using \( \int_D \psi(x)\,dx = 1 \), we have
\begin{align}\label{kik4}
v_1(t,\psi) + v_2(t,\psi) 
\geq \left( \min_{x \in \bar{D}} |v_1(t,x)| + \min_{x \in \bar{D}} |v_2(t,x)| \right) 
\int_D \psi(x)\,dx 
\geq 2 \min\left\{ \min_{x \in \bar{D}} |v_1(t,x)|, \min_{x \in \bar{D}} |v_2(t,x)| \right\}.
\end{align}
Combining~\eqref{kik3} and~\eqref{kik4} yields the upper estimate~\eqref{ak71b}.
\end{proof}

\section{Estimates of the quenching probability when $\rho_{11}=\rho_{21}:=\rho_{1}$ and $\rho_{12}=\rho_{22}:=\rho_{2}$}	\label{s5}
In this subsection, our primary objective is to derive both lower and upper bounds for the probability of quenching, denoted by $\mathbb{P}\{\tau_q < \infty\}$. To achieve this, we utilize tail probability estimates for exponential functionals of fractional Brownian motion (fBM). These estimates, which play a crucial role in our analysis, have been rigorously developed in the works \cite{dung,dung2}. By leveraging these results, we are able to obtain meaningful probabilistic bounds that characterize the likelihood of finite-time quenching in the considered stochastic system ~\eqref{rand1a}--\eqref{rand1c}, or equivalently  for system~\eqref{zeq1}--\eqref{zeq4}.

\subsection{Upper bound}
Indeed, in order to establish rigorous lower bounds for the quenching probability, we rely on the following key result. 
\begin{theorem}(\cite[Theorem 3.1]{dung2})\label{thme6.1} Let $\{X_{t}\}_{t \in [0,T]}$ be a continuous stochastic process in $\mathbb{D}^{1,2}.$ Assume that one of the following two conditions holds:
	\begin{itemize}
		\item[(i)] $\displaystyle\sup_{s \in [0,T]} \int_{0}^{T} |\mathcal{D}_{r}X_{s}|^{2}dr \leq M^{2},\ \mathbb{P}\text{-a.s.},$ \nonumber\\
		\item[(ii)] $\displaystyle\int_{0}^{T} \mathbb{E} \left[ \displaystyle\sup_{s \in [0,T]} |\mathcal{D}_{r}X_{s}|^{2}| \mathcal{F}_{r} \right] dr \leq \mathcal{M}^{2},$ $\mathbb{P}$\text{-a.s.},		
	\end{itemize}
    where $\mathcal{M}$ is a non-random constant.
    
	Then the tail probability of the exponential functional satisfies
	\begin{align}
		\mathbb{P}\left( \int_{0}^{T} \exp\{X_{s}\} ds \geq x\right) \leq 2 \exp\Big\{-\frac{(\ln x- \ln \theta)^{2}}{2 \mathcal{M}^{2}}\Big\},\ x>\theta, \nonumber 
	\end{align}
	where $\theta = \displaystyle\int_{0}^{T} \mathbb{E}[e^{X_{s}}]ds.$
\end{theorem}
Building upon the preparatory estimates and theoretical tools introduced above, we are now in a position to present our first main result in this direction. 
\begin{theorem}\label{thm6}
    Let
\[
m_{0}(T) = \int_{0}^{T} \exp\{\sigma s\} \, \mathbb{E}\left( \exp\left\{ \rho_{1} W(s) + \rho_{2} B^{H}(s) \right\} \right) ds,
 \]
 for $\sigma:= 3(\chi + k^{2}).$
Then, for any $T > 0$ satisfying
\[
\frac{E^{3}(0)}{12 \tilde{\lambda}} > m_{0}(T),
\]
the following inequality holds:

		\begin{align}
			\mathbb{P}\left\{\tau_q \leq T \right\}\leq\mathbb{P}\left\{\tau^{\ast} \leq T \right\} \leq   2 \exp\left\lbrace -\frac{(\ln \left( \frac{E^{3}(0)}{12 \tilde{\lambda}}\right) - \ln(m_{0}(T)))^{2}}{2(2\rho_{1}^{2}T +2\rho_{2}^{2}T^{2H})^{2}}\right\rbrace, \nonumber 
		\end{align}
		where the quantities $k^{2}$, $E(0)$, and $\tilde{\lambda}$ are specified in Theorem~\ref{thm4.1}, while $\rho_{1}$ and $\rho_{2}$ are defined in equation~\eqref{eq1}.
\end{theorem}
\begin{proof}
	For each $t \geq 0$ consider the stochastic process
	$$Z_{t}=\sigma t+\rho_{1}W(t)+\rho_{2}B^{H}(t).$$
	Note that $\{Z_{t}\}_{t \geq 0}$ satisfies condition $(i)$ in Theorem \ref{thme6.1}. Indeed,
	\begin{align}
		\mathcal{D}_{r}Z_{t} &= \rho_{1}+\rho_{2}\mathcal{D}_{r}\left(\int_{0}^{t} K^{H}(t,s)dW(s) \right)= \rho_{1}+\rho_{2}  K^{H}(t,r),\ \mbox{for}\ r \leq t,  \nonumber 
	\end{align}
	and $\mathcal{D}_{r}Z_{t} =0,$ for $r>t.$ Further, we have
	\begin{align}
		\int_{0}^{T} |\mathcal{D}_{r}Z_{t}|^{2}dr= \int_{0}^{t} \left[  \rho_{1}+\rho_{2}K^{H}(t, r)\right]^{2} dr &\leq    2\rho^{2}_{1}t+2\rho_{2}^{2}\int_{0}^{t}   (K^{H}(t,r))^{2}dr \nonumber\\
		&= 2\rho^{2}_{1}t+2\rho_{2}^{2}\mathbb{E}|B^{H}(t)|^{2}= 2\rho^{2}_{1}t+2\rho_{2}^{2}t^{2H},\nonumber 
	\end{align}
	hence
	\begin{align}\label{g1}
		\sup_{t \in [0,T]}\int_{0}^{T} |\mathcal{D}_{r}Z_{t}|^{2} dr \leq 2\rho_{1}^{2}T +2\rho_{2}^{2}T^{2H}=\mathcal{M}(T)<\infty.
	\end{align}
Since, by Theorem~\ref{thm4.1}, we have $\tau_q \leq \tau^{\ast}$, where $\tau^{\ast}$ is defined by \eqref{ST2b}, it then follows from Theorem~\ref{thme6.1} that:

	\begin{align}
		\mathbb{P}\left\{ \tau_q \leq T\right\}\leq\mathbb{P}\left\{ \tau^{\ast} \leq T\right\}&= \mathbb{P}\left( \int_{0}^{T} \exp\{Z_{s}\} ds\geq \frac{E^{3}(0)}{12 \tilde{\lambda}} \right) \nonumber\\
		&\leq  2 \exp\left\lbrace -\frac{\left( \ln \left( \frac{E^{3}(0)}{12 \tilde{\lambda}}\right) - \ln(m_{0}(T))\right)^{2}}{2(\rho_{1}^{2}T +2\rho_{2}^{2}T^{2H})^{2}}\right\rbrace, \nonumber 
	\end{align}
	for $m_{0}(T)=\displaystyle\int_{0}^{T} \mathbb{E}\left(\exp\{Z_{s}\} \right) ds,$
	 which completes the proof. 
\end{proof}
The following theorem provides upper bounds for the tail probability of $\tau^{\ast}$, which in turn yields an upper bound for the quenching probability, under the general dependent structure of Brownian motion and fractional Brownian motion (fBm).

\begin{theorem}\label{thm6t}

	\begin{itemize}
		\item[ ] {\bf Case 1:} Suppose $W(t)$ and $B^{H}(t)$ are dependent in the sense that they are related via the integral representation \eqref{nk71a} then
		\begin{align}
			\mathbb{P}\{\tau^{\ast}\leq T\} &\leq \frac{6 \tilde{\lambda}}{E^{3}(0)}\Bigg[ \frac{\exp\{\left( \rho^{2}_{1}+3(\chi+k^{2})\right) T\}-1}{\left( \rho^{2}_{1}+3(\chi+k^{2})\right) } +  \int_{0}^{T} \exp\{3(\chi+k^{2})s+2 \rho_{2}^{2}s^{2H} \}ds\Bigg]. \nonumber 
		\end{align}	
		\item[ ] {\bf Case 2:} If $W(t)$ and $B^{H}(t)$ are independent, then 
		\begin{align}
			\mathbb{P}\{\tau^{\ast}\leq T\}\leq  \frac{12 \tilde{\lambda}}{E^{3}(0)}\int_{0}^{T} \exp \left\lbrace \left( \frac{\rho^{2}_{1}}{2} +3(\chi+k^{2})\right) s+\frac{\rho_{2}^{2}}{2} s^{2H}\right\rbrace ds. \nonumber
		\end{align}	
	\end{itemize}
Recall that the quantities $k^{2}$, $E(0)$, and $\tilde{\lambda}$ are specified in Theorem~\ref{thm4.1}, while $\rho_{1}$ and $\rho_{2}$ are defined in equation~\eqref{eq1}.
\end{theorem}
\begin{proof}
	\textbf{Case 1:} Recall that since $W(t), B^H(t)$ are related via the integral representation \eqref{nk71a} then $$B^{H}(t)=\int_{0}^{t} K^{H}(t,s)dW(s),$$ where the kernel $K^{H}(\cdot, \cdot)$ is given by \eqref{gsf58} and $W(t)$ is  defined in the same probability space and adapted to the same filtration as $B^H(t).$ Since $\tau_q \leq \tau^{\ast},$ where $\tau^{\ast}$ is given in \eqref{ST2b}, and by using H\"older's and Markov's  inequalities, we deduce
	\begin{align}
		\mathbb{P}\{\tau^{\ast} \leq T\}&= \mathbb{P}\left( \int_{0}^{T} \exp\{\rho_{1} W(s)+\rho_{2} B^{H}(s)+3(\chi+k^{2})s\} ds\geq \frac{E^{3}(0)}{12 \tilde{\lambda}} \right) \nonumber\\
		& \leq \mathbb{P}\Bigg[  \int_{0}^{T} \left( \exp\{2\rho_{1} W(s)+3(\chi+k^{2})s\} ds\right)^{\frac{1}{2}} \nonumber\\
		&\qquad \times \int_{0}^{T} \left( \exp\{2\rho_{2} B^{H}(s)+3(\chi+k^{2})s\} ds\right)^{\frac{1}{2}}\geq \frac{E^{3}(0)}{12 \tilde{\lambda}} \Bigg] \nonumber\\
		& \leq \mathbb{P}\Bigg[  \int_{0}^{T}  \exp\{2\rho_{1} W(s)+3(\chi+k^{2})s\} ds\geq \frac{E^{3}(0)}{6 \tilde{\lambda}}\Bigg] \nonumber\\
		&\qquad +\mathbb{P}\Bigg[ \int_{0}^{T}  \exp\{2\rho_{2} B^{H}(s)+3(\chi+k^{2})s\} ds\geq \frac{E^{3}(0)}{6 \tilde{\lambda}} \Bigg] \nonumber\\
		&\leq \frac{1}{\frac{E^{3}(0)}{6 \tilde{\lambda}}}\Bigg\{\mathbb{E}\left[ \int_{0}^{T} \exp\{2\rho_{1} W(s)+3(\chi+k^{2})s\} ds \right]\nonumber\\
		&\qquad+ \mathbb{E}\left[ \int_{0}^{T} \exp\{2\rho_{2} B^{H}(s)+3(\chi+k^{2})s\} ds \right]\Bigg\} \nonumber\\
		&\leq \frac{6 \tilde{\lambda}}{E^{3}(0)}\Bigg[ \int_{0}^{T}  \exp\{\rho_{1}^{2}s+3(\chi+k^{2})s\} ds \nonumber\\
		&\qquad+ \int_{0}^{T} \exp\{3(\chi+k^{2})s\} \mathbb{E}\left[ \exp\{2\rho_{2} B^{H}(s)\}\right] ds \Bigg].	\label{c2}
	\end{align}
	Also by virtue of \eqref{f2} we have
	\begin{align}
	\mathbb{E}\left[\exp\{2\rho_{2} B^{H}(s)\}  \right] &= \mathbb{E}\left[ \exp\{2\rho_{2} \int_{0}^{s} K^{H}(s,r)dW(s)\}\right] = \exp\{4 \rho_{2}^{2}\int_{0}^{s} (K^{H}(s,r))^{2} ds\} \nonumber\\
		&= \exp\{4 \rho_{2}^{2}\mathbb{E}(|B^{H}_{s}|^{2})\}=\exp\{4 \rho_{2}^{2}s^{2H} \},	\nn
	\end{align}
	hence \eqref{c2} yields
	\begin{align}
		\mathbb{P}\{\tau^{\ast}\leq T\} &\leq \frac{6 \tilde{\lambda}}{E^{3}(0)} \Bigg[ \int_{0}^{T} \exp\{\left( \rho^{2}_{1}+3(\chi+k^{2})\right) s\}ds +  \int_{0}^{T} \exp\{3(\chi+k^{2})s+4 \rho_{2}^{2}s^{2H} \}ds\Bigg] \nonumber\\
		&=\frac{6 \tilde{\lambda}}{E^{3}(0)} \Bigg[ \frac{\exp\{\left( \rho^{2}_{1}+3(\chi+k^{2})\right) T\}-1}{\left( \rho^{2}_{1}+3(\chi+k^{2})\right) } +  \int_{0}^{T} \exp\{3(\chi+k^{2})s+4 \rho_{2}^{2}s^{2H} \}ds\Bigg]. \nonumber
	\end{align}
	\textbf{Case 2:} If $W(t)$ and $B^{H}(t)$ are independent, then by using Markov's inequality,  we obtain 
	\begin{align}
		\mathbb{P}\{\tau^{\ast}\leq T\}&= \mathbb{P}\left( \int_{0}^{T} \exp\{\rho_{1} W(s)+\rho_{2} B^{H}(s)+3(\chi+k^{2})s\} ds\geq \frac{E^{3}(0)}{12\tilde{\lambda}}  \right) \nonumber\\
		&\leq \frac{12 \tilde{\lambda}}{E^{3}(0)} \int_{0}^{T} \mathbb{E}\left(\exp\{\rho_{1} W(s)+3(\chi+k^{2})s\}\right) \mathbb{E}\left(\exp\{\rho_{2} B^{H}(s)\} \right)  ds \nonumber\\
		&= \frac{12 \tilde{\lambda}}{E^{3}(0)}\int_{0}^{T} \exp \left\lbrace \left( \frac{\rho^{2}_{1}}{2}+3(\chi+k^{2})\right) s+ \frac{\rho_{2}^{2}}{2}s^{2H}\right\rbrace  ds, \nonumber 
	\end{align}
	which completes the proof.
    \end{proof}	

    \subsection{Lower bound}
    In this subsection, we derive a lower bound for the quenching probability of the weak solution $z = (z_1, z_2)^{\top}$ of the system \eqref{zeq1}--\eqref{zeq4}. To this end, we make use of the following key result.
    \begin{theorem}(\cite[Theorem 3.1 ]{dung})\label{thm6.4}
    	Suppose that the stochastic process $\{X_{t} \}_{t \geq 0}$ is $\mathcal{F}_t$-adapted and satisfies the following properties:   
    \begin{itemize}	
    	\item [(i)] $\displaystyle\int_{0}^{\infty} \mathbb{E}\left[ \exp\{-\gamma s+ \delta X_{s}\}\right]  ds < +\infty,$ for some $\gamma\in \mathbb{R}$ and $\delta>0.$
    	\item [(ii)] For each $t \geq 0, \ X_{t} \in \mathbb{D}^{1,2}.$
    	\item [(iii)] There exists a function $f : \mathbb{R}^{+} \rightarrow \mathbb{R}^{+}$ such that $\displaystyle\lim_{t \rightarrow \infty} f(t)= +\infty$ and for each $U>0,$
    	\begin{align*}
    		\sup\limits_{t\geq 0} \frac{\sup\limits_{s\in [0,t]} \int_{0}^{s} |\mathcal{D}_{\theta_{1}}X_{s}|^{2} d\theta_{1}}{\left( \ln \left( U+1\right)+f(t) \right)^{2}} \leq L_{U}< +\infty, \ \ \mathbb{P}\text{-a.s.}
    	\end{align*}
    \end{itemize} Then, there holds 
    	\begin{align}
    		\mathbb{P} \left[ \int_{0}^{\infty} \exp\{-\gamma s+\delta X_{s} \} ds < U \right] \leq \exp \left\lbrace -\frac{(m_{U} -1)^{2}}{2 \delta^{2}L_{U}}\right\rbrace,
    	\end{align}
    	where 
    	\begin{align*}
    		m_{U}=\mathbb{E}\left[\sup\limits_{t \geq 0} \frac{\ln \left(  \int_{0}^{t} \exp\{-\gamma s+\delta X_{s} \} ds+1 \right) +f(t)}{\ln \left( U+1\right)+f(t)}\right]\geq 1.
    	\end{align*}
    \end{theorem}
    We are now in a position to present our principal estimation result, which pertains to establishing a rigorous lower bound for the probability of quenching.
    \begin{theorem}\label{thm6.5}
    	For any $H<\alpha<1$ a lower bound for the probability of finite-time quenching of the solution $z = (z_1, z_2)^{\top}$ of the system \eqref{zeq1}--\eqref{zeq4} is given by 
    	\begin{align}\label{PB1}
    	&\mathbb{P} \left\{\tau_q< \infty \right\}\nonumber\\ &\geq  
    	1-\exp \left\lbrace -\frac{\alpha^{2}(L_{1}(\alpha) -1)^{2}}{\rho_{1}^{2}(2\alpha-1)^{2-\frac{1}{\alpha}}\ln(U+1)^{\frac{1}{\alpha}-2}+2 \rho_{2}^{2}\alpha^{2}\ln(U+1)^{\frac{2H}{\alpha}-2}\left( \frac{\alpha-H}{\alpha}\right)^{2-\frac{2H}{\alpha}}}\right\rbrace, 
    	\end{align}
    	where $U:=\frac{E^{3}(0)}{12 \tilde{\lambda}}$ and 
        \begin{align}\label{N1}
    	L_{1}(\alpha):= \mathbb{E}\left[\sup\limits_{t \geq 0} \frac{\ln \left(  \int_{0}^{t} \exp\{3(\chi+k^{2})s+\rho_{1}W(s)+\rho_{2}B^{H}(s) \} ds+1 \right) +t^{\alpha}}{\ln \left( \frac{E^{3}(0)}{12 \tilde{\lambda}}+1\right)+t^{\alpha}}\right]. 
    	\end{align}
    	Also  the quantities $k^{2}$, $E(0)$, and $\tilde{\lambda}$ are specified in Theorem~\ref{thm4.1}, while $\rho_{1}$ and $\rho_{2}$ are defined in equation~\eqref{eq1}. 
    \end{theorem}
    \begin{proof}
   
    	We first note that since $\sigma=3 (\chi+k^2)>0$ then for any $t \geq 0$ 
        \begin{align*}
	\int_{0}^{t} \exp\{\sigma s-2\sigma s-\rho_{1}^{2}s +\rho_{1}W(s)-\rho_{2}^{2}s^{2H}+\rho_{2}B^{H}(s)\}ds \leq \int_{0}^{t} \exp\{\sigma s +\rho_{1}W(s)+\rho_{2}B^{H}(s)\}ds,
	\end{align*}
     which implies 
\begin{align}
	&\mathbb{P} \left( \int_{0}^{\infty} \exp\{\sigma s +\rho_{1}W(s)+\rho_{2}B^{H}(s)\}ds<U \right) \nonumber\\
	&\qquad \qquad \qquad \leq  \mathbb{P} \left( \int_{0}^{\infty} \exp\{\sigma s-2 \sigma s-\rho_{1}^{2}s +\rho_{1}W(s)-\rho_{2}^{2}s^{2H}+\rho_{2}B^{H}(s)\}ds<U \right),
    \label{mmnk16}\end{align}	
for any $U>0.$

Next, we consider again the stochastic process 
    	\begin{align}
    		Z_{t}&=\sigma t + \rho_{1} W(t)+\rho_{2} B^{H}(t), \nonumber
    	\end{align}
    	then $\{Z_{t}\}_{t \geq 0},$ satisfies the assumptions  of Theorem \ref{thm6.4} for $\gamma=0$ and $\delta=1.$ 
        
        Indeed, due to \eqref{mmnk16}
    	\begin{align}
    		\int_{0}^{\infty}\mathbb{E} \left(  e^{Z_{s}}\right) ds &\leq \int_{0}^{\infty} e^{-\rho_{1}^{2}s-\sigma s - \rho_{2}^{2}s^{2H}} \mathbb{E}\left( e^{\rho_{1}W(s)+\rho_{2}B^{H}(s)}\right)  ds\nonumber\\
    		&=\int_{0}^{\infty} e^{-\sigma s} e^{-\rho_1^2s+\frac{\rho_1^2s}{2}} e^{{-\rho_2^2s^{2H}+\frac{\rho_1^2s^{2H}}{2}}}ds\nonumber\\
            &=\int_{0}^{\infty} e^{-\sigma s-\frac{\rho_1^2s}{2}-\frac{\rho_2^2s^{2H}}{2}} ds<\infty, \nonumber  
    	\end{align}
       using also \eqref{f2}.  Also for $r \leq t,$ we have
    	\begin{align}
    		\mathcal{D}_{r}Z_{t} = \rho_{1}+\rho_{2}\mathcal{D}_{r}B^{H}(t) &= \rho_{1}+\rho_{2}\mathcal{D}_{r}\left( \int_{0}^{t}K^{H}(t, s)dW(s)\right) = \rho_{1}+\rho_{2}K^{H}(t, r).  \nonumber
    	\end{align}
    	Furthermore by \eqref{g1}, we have
    	\begin{align} 
    		\sup_{s \in [0,t]}\int_{0}^{s} |\mathcal{D}_{r}Z_{s}|^{2} \leq 2\rho_{1}^{2}t +2\rho_{2}^{2}t^{2H}<\infty.\nonumber 
    	\end{align}
    	Note also that $\{Z_{t}\}_{t \geq 0}$  satisfies the condition (iii) of Theorem \ref{thm6.4}  for  $f(t)=t^{\alpha}, \alpha>H$ and $U>0.$ In particular, 
    	\begin{align}
    		L_{U}&:= \sup_{t \geq 0} \frac{2\rho_{1}^{2}t +2\rho_{2}^{2}t^{2H}}{(\ln(U+1)+t^{\alpha})^{2}}\nonumber\\
    		&\leq \sup_{t \geq 0} \frac{2\rho_{1}^{2}t }{(\ln(U+1)+t^{\alpha})^{2}}+\sup_{t \geq 0} \frac{2\rho_{2}^{2}t^{2H}}{(\ln(U+1)+t^{\alpha})^{2}}\nonumber\\	
    	&\leq  \frac{\rho_{1}^{2}(2\alpha-1)^{2-\frac{1}{\alpha}}}{\alpha^{2}}\ln(U+1)^{\frac{1}{\alpha}-2}+2 \rho_{2}^{2}\ln(x+1)^{\frac{2H}{\alpha}-2}\left( \frac{\alpha-H}{\alpha}\right)^{2-\frac{2H}{\alpha}}. \nonumber  
    	\end{align}
    	Finally,   Theorem \ref{thm6.4} for $U=\frac{E^{3}(0)}{12 \tilde{\lambda}}$ yields
    	\begin{align}
    		&\mathbb{P} \left[ \int_{0}^{\infty} \exp\{3(\chi+k^{2})s+\rho_{1}W(s)+\rho_{2}B^{H}(s) \} ds < \frac{E^{3}(0)}{12 \tilde{\lambda}} \right] \nonumber\\
    		&\qquad\leq \exp \left\lbrace \frac{-\alpha^{2} (L_{1}(\alpha) -1)^{2}}{\rho_{1}^{2}(2\alpha-1)^{2-\frac{1}{\alpha}}\ln(U+1)^{\frac{1}{\alpha}-2}+2 \alpha^{2} \rho_{2}^{2}\ln(U+1)^{\frac{2H}{\alpha}-2}\left( \frac{\alpha-H}{\alpha}\right)^{2-\frac{2H}{\alpha}}}\right\rbrace:= \mathcal{N}(U,\alpha), \nonumber 
    	\end{align}
    	where $L_{1}(\alpha)$ is given in \eqref{N1}.

    Hence by \eqref{ST2b} we deduce that
    \begin{align*}
    \mathbb{P} \left\{\tau_q< \infty \right\}\leq \mathbb{P} \left\{\tau^*< \infty \right\}&= 1- \mathbb{P}\{\tau^*=\infty\}\\&\geq 1-\mathcal{N}(U,\alpha),
   \end{align*}   
        thus the desired bound is provided by \eqref{PB1} is derived. This completes the proof.
    \end{proof}

    \subsection{An interesting special case} 

    In this subsection, we consider the case where \( \frac{3}{4} < H < 1 \), and the processes \( \{W(t): t\geq 0\} \) and \( \{B^{H}(t): t\geq 0\} \) are independent, with a common coupling coefficient \( \rho_1 = \rho_2 = \rho \). In this setting, finite-time quenching occurs almost surely for the weak solution \( v = (v_1, v_2)^{\top} \) of system~\eqref{rand1a}--\eqref{rand1c}. This result is obtained by adapting the approach in~\cite[Subsection~5.2]{KNY24}; see also~\cite[Subsection~5.1]{DKN22} and~\cite[Section~3]{DNMK23} for related analyses. 

According to~\cite[Theorem~1.7]{cher2001}, the process \(  M(t):= W(t) + B^H(t) \) is equivalent in law to a standard Brownian motion $\widetilde{B}(t)$. Here, equivalence refers to the fact that the probability laws of \( M \) and \( \widetilde{B} \) coincide on the space \( (C[0,T], \mathcal{B}) \), where \( C[0,T] \) is the space of continuous functions on \([0,T]\), and \( \mathcal{B} \) is the \(\sigma\)-algebra generated by the cylinder sets.

\begin{theorem}\label{thm6.6}
Let \( \frac{3}{4} < H < 1 \). Then the solution \( v = (v_1, v_2)^{\top} \) to problem~\eqref{rand1a}--\eqref{rand1c}, as well as the solution \( z = (z_1, z_2)^{\top} \) to~\eqref{zeq1}--\eqref{zeq4}, quenches in finite time \( \tau_q < \infty \) almost surely.
\end{theorem}

\begin{proof}
From identity~\eqref{ST2b}, we have
\begin{align*}
\mathbb{P}\{ \tau^{\ast} = \infty \} 
&= \mathbb{P}\left( \int_{0}^{t} \exp\left\{ \rho W(s) + \rho B^H(s) + \sigma s \right\} ds < \frac{E^3(0)}{12 \tilde{\lambda}},\ \text{for all } t \geq 0 \right) \\
&= \mathbb{P}\left( \int_{0}^{\infty} \exp\left\{ \rho W(s) + \rho B^H(s) + \sigma s \right\} ds \leq \frac{E^3(0)}{12 \tilde{\lambda}} \right) \\
&= \mathbb{P}\left( \int_{0}^{\infty} \exp\left\{ \rho \widetilde{B}(s) + \sigma s \right\} ds \leq \frac{E^3(0)}{12 \tilde{\lambda}} \right) \\
&= \mathbb{P}\left( \int_{0}^{\infty} \exp\left\{ 2 \widetilde{B}\left( \frac{\rho^2 s}{4} \right) + \sigma s \right\} ds \leq \frac{E^3(0)}{12 \tilde{\lambda}} \right),
\end{align*}
where the last step uses the scaling property of Brownian motion.

Now, setting the change of variables \( t = \frac{\rho^2 s}{4} \) and defining \( \nu := \frac{2 \sigma}{\rho^2} \), see also \cite{DKN22, DNMK23, KNY24}, we obtain
\begin{align}\label{gsf2}
\mathbb{P}\{ \tau^{\ast} = \infty \}
&= \mathbb{P} \left( \frac{4}{\rho^2} \int_{0}^{\infty} \exp\left\{ 2 \widetilde{B}_t^{(\nu)} \right\} dt \leq \frac{E^3(0)}{12 \tilde{\lambda}} \right) \nonumber \\
&= \mathbb{P} \left( \int_{0}^{\infty} \exp\left\{ 2 \widetilde{B}_t^{(\nu)} \right\} dt \leq \frac{\rho^2 E^3(0)}{48 \tilde{\lambda}} \right),
\end{align}
where \( \widetilde{B}_t^{(\nu)} := \widetilde{B}(t) + \nu t \) is a Brownian motion with linear drift.

Since \( \nu > 0 \), the law of the iterated logarithm (cf.~\cite[Theorem~2.3]{A95}, \cite[Theorem~9.23]{KS91}) implies
\begin{align*}
\liminf_{t \to \infty} \frac{\widetilde{B}(t)}{t^{1/2} \sqrt{2 \log(\log t)}} &= -1, \quad \mathbb{P}\text{-a.s.}, \\
\limsup_{t \to \infty} \frac{\widetilde{B}(t)}{t^{1/2} \sqrt{2 \log(\log t)}} &= +1, \quad \mathbb{P}\text{-a.s.}
\end{align*}
Thus, for any sequence \( t_n \to \infty \), we have the asymptotic behaviour
\[
\widetilde{B}_{t_n} \sim \alpha_n t_n^{1/2} \sqrt{2 \log(\log t_n)}, \quad \text{with } \alpha_n \in [-1,1],
\]
which implies that
\[
\int_{0}^{\infty} \exp\left\{ 2 \widetilde{B}_t^{(\nu)} \right\} dt = \infty.
\]
From~\eqref{gsf2}, we then conclude that
\[
\mathbb{P}[\tau^* = \infty] = 0,
\]
and hence
\[
\mathbb{P}[\tau^* < \infty] = 1.
\]
Since \( \tau_q < \tau^* \), it follows that
\[
\mathbb{P}[\tau_q < \infty] = 1.
\]
Therefore, the solution to problem~\eqref{rand1a}--\eqref{rand1c}, as well as the solution to~\eqref{zeq1}--\eqref{zeq4}, quenches in finite time almost surely.
\end{proof}
   
\section{Numerical solution}\label{ns}

\subsection{Finite elements approximation}
	In this section, we present a numerical investigation of problem~\eqref{b1}--\eqref{b1d} in the one-dimensional setting. To this end, we employ a finite element semi-implicit Euler scheme for the time discretization; see, for example, \cite{Lord}.

Our study primarily focuses on homogeneous and nonhomogeneous Robin boundary conditions imposed at the endpoints $x = 0$ and $x = 1$. However, some of the numerical experiments also explore the case of homogeneous Dirichlet boundary conditions, which was not addressed in the analytical treatment provided in the previous sections.

Recall that the noise term considered here is multiplicative and takes the form $\sigma(u_i) \, dN_i(t)$, where $\sigma(u_i) = 1 - u_i$ for $i = 1,2$. In particular, we assume that the noise is given by the expression
\[
\sigma(u_i) \, dN_i(t) = k_{i1}(1 - u_i) \, dW(t) + k_{i2}(1 - u_i) \, dB^H(t),\; i=1,2,
\]
representing a combination of standard Brownian motion $W(t)$ and fractional Brownian motion $B^H(t)$, with coefficients $k_{i1}, k_{i2}, i=1,2$ determining the intensity of each component.
	
We apply
a discretization in $[0,T]\times [0,1]$, $0\leq t\leq T$, $0\leq x\leq 1$ with
$t_n=n\delta t$, $\delta t=\left[{T}/{N}\right]$  for $N$ the number of time steps and we also
 introduce the grid points in $[0,1]$, $x_j = j\delta x$, for $\delta x = 1/M$ and $j = 0,1,\ldots, M$.
	
	 Then we proceed with a finite element approximation for problem \eqref{b1}--\eqref{b1d}.
In particular, let $\Phi_j$, $j=1,\ldots, M-1,$ denote the standard linear $B-$ splines on the interval $[0,\,1]$
\begin{eqnarray}
\Phi_j=\left\{\begin{array}{ccc}
 \frac{y-y_{j-1}}{\delta y},\quad y_{j-1}\leq y \leq y_j, \\
 \frac{y_{j+1}\,-y}{\delta y},\quad y_{j}\leq y \leq y_{j+1}, \\
 0,\quad \mbox{elsewhere in}\quad [0,\,1],
\end{array} \right.
\end{eqnarray}
for $j=1,2,\ldots,M-1$. We then set 
$u_i(t,x)=\sum_{j=1}^{M-1} {a}_{{u_i}_j}(t) \Phi_j(x)$, 
$t\geq 0$, $0\leq x\leq 1$ for $i=1,2.$
	
Substituting the expansion for $u = (u_1, u_2)$ into equation~\eqref{b1} and applying the standard Galerkin method, namely, multiplying by the test functions $\Phi_\ell$ for $\ell = 1, 2, \ldots, M-1$ and integrating over the interval $[0,1],$ we obtain a system of equations for the coefficients $a_{{u_i}_j}$, given by:

\begin{eqnarray}\hspace{-.1cm}\label{eqfe}
   \sum_{j=1}^{M-1}  {\dot{a}}_{{u_i}_j}  (t)<\Phi_j(x),\,\Phi_\ell(x)> &= &
 -\sum_{j=1}^{M-1} {a}_{{u_i}_j}(t)\left<\Phi'_j(x),\,\Phi'_\ell(y)\right> \nonumber\\
  &&+  \left<F_i\left(\sum_{j=1}^{M-1} {a}_{{u_1}_j}(t) \Phi_j(x), \sum_{j=1}^{M-1} {a}_{{u_2}_j}(t) \Phi_j(x)\right),\,\Phi_\ell(x)\right>,\quad\quad
 \nonumber\\
 &&+  \left<\sigma \left(\sum_{j=1}^{M-1} {a}_{{u_i}_j}(t) \Phi_j(x)\right) dN_i(t),\,\Phi_\ell(x)\right>,\quad\quad
\end{eqnarray}
where $<\cdot,\cdot>$  stands for the inner producr in  $L^2(0,1)$ and in our case 
$F_i(s_1,s_2)=\frac{\lambda_{i1}}{(1-s_1)^2}+\frac{\lambda_{i2}}{(1-s_2)^2}$, $\sigma(s_i)=(1-s_i)$, $i=1,2$.
	
	 Setting $a_{u_i}=[a_{{u_i}_1},\,a_{{u_i}_2},\ldots, a_{{u_i}_{M-1}}]^T$
    the system of equations for the coefficients column vector ${a_{u_i}}$ takes the form
    \begin{eqnarray}
A_i {\dot{a}_{u_i}}(t)&=&-B_i a_{u_i}(t) +b_i(t) + {b_i}_s(t),\nonumber
\end{eqnarray}
 for 
 \begin{eqnarray}
 b_i({t}):=b_i(u_1, u_2)=\left\{\left<F_i \left(\sum_{j=1}^{M-1} {a}_{{u_i}_j}(t) \Phi_j(x),\sum_{j=1}^{M-1} {a}_{{u_i}_j}(t) \Phi_j(x)\right),\,\Phi_\ell(x), 
 \right>\right\},\\ 
 {b}_{is}(t):={b_i}({u_i},\,\Delta {N_i}(t))=
 \left\{\left<\sigma \left(\sum_{j=1}^{M-1} {a}_{{u_i}_j}(t) \Phi_j(x)\right) \Delta N_i(t),\,\Phi_{\ell}(x)\right>\right\},
 \end{eqnarray} 
  with the ${(M-1)\times(M-1)}$ matrices $A_i,B_i$  having the standard form (see ~\cite{DKN22}).
  
  The latter term comes from the corresponding form of the noise $N_i, i=1,2$ (see \eqref{cnt}) and $d {N_i}(t)\simeq \Delta {N_i}_h(t)={N_i}_h(t+\delta t)-{N_i}_h(t)$ for ${N_i}_h(t)$ the finite sum giving the discrete approximation of ${N_i}(t)$, $i=1,2$. Moreover, we approximate the white noise by the method followed in \cite{DKN22} (see also \cite{DNMK23, KNY24, Lord},)

  We then apply a semi-implicit Euler method in time by taking
\[ A_i{\dot{a}_{u_i}(t_{n})}\simeq A_i\left({a_{u_i}^{n+1}-a_{u_i}^{n}}\right)/({\delta t})=
 -B_i a_{u_i}^{n+1} + b_i({u_1^{n},u_2^{n}})+{b_{is}}(u_i^n,\,\Delta {N_i}_h^n),\]    
   or 
   \[\left( A_i+\delta t B_i\right)a_{u_i}^{n+1} =a_{u_i}^{n}+\delta t\, b_i({u_1^{n},u_2^{n}})+\delta t \,{b_i}_s({u_i}^n,\,\Delta {N_i}_h^n).\]
  Finally, the corresponding algebraic system for the $a_{u_i}^n$'s after some manipulation becomes
 \begin{eqnarray*}\label{numscha1}
 a_{u_1}^{n+1}=\left(A_1+\delta t B_1 \right)^{-1}\left[a_{u_1}^{n} + \delta t\, b_i(u_1^{n},u_2^{n})+ 
 \delta t\, {b_i}_s(u_1^{n})\right], \\
  a_{u_2}^{n+1}=\left(A_2+\delta t B_2 \right)^{-1}\left[a_{u_2}^{n} + \delta t\, b_i(u_1^{n},u_2^{n})+
   \delta t\, {b_i}_s(u_2^{n})\right],
\end{eqnarray*}
 for $a_u^1= (a_{u_1}^1, a_{u_2}^1)$  being determined by the initial conditions.  

 \subsection{Simulations} 
In this section, we present a set of illustrative simulations of the problem. A more in-depth numerical analysis based on the employed numerical scheme is beyond the scope of the current study and is left for future work.

Initially we present a realization of the numerical solution of problem~\eqref{b1}--\eqref{b1d} in Figure~\ref{Fig_sim1}
 (a)   for $H=0.6$,
 $\lambda_{ij}=.08$, $k_{i1}=.008$, $k_{i2}=.008$ ,  $i,j=1,2$  initial condition $u_i(x,0)=c\,x(1-x)$ for $c=0.1$, $i=1,2$ and $\beta=\beta_c=1$. The performed simulation clearly demonstrates the occurrence of finite-time quenching. In a different realization with the same set of parameters, illustrated in Figure~\ref{Fig_sim1}(b), we plot the maximum values of the solutions $u_1$ and $u_2$ at each time step. A similar quenching behaviour is observed in this case as well. Notably, the components $u_1$ and $u_2$ appear to quench almost simultaneously.

\begin{figure}[H]\vspace{-4cm}\hspace{-1cm}
   \begin{minipage}{0.48\textwidth}
     \centering
     \includegraphics[width=1.2\linewidth]{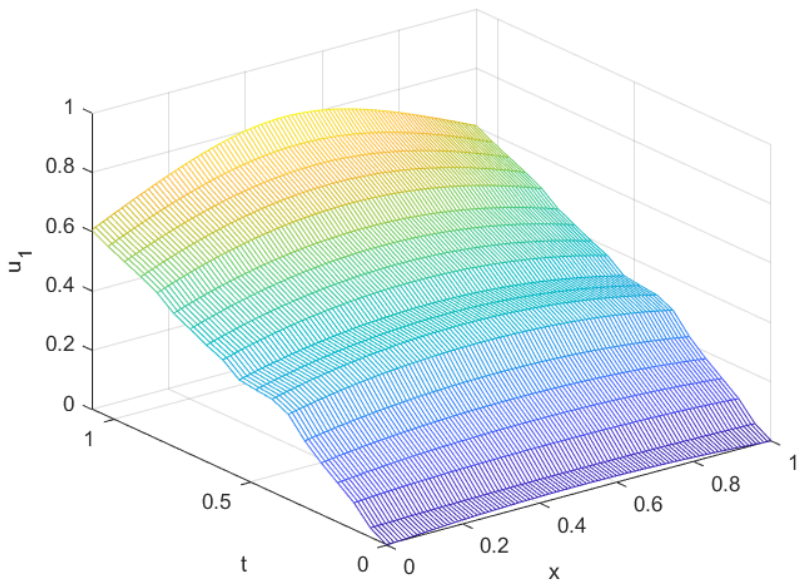}
   \end{minipage}\hfill
   \begin{minipage}{0.48\textwidth}\hspace{-2cm}
     \centering
     \includegraphics[width=1.2\linewidth]{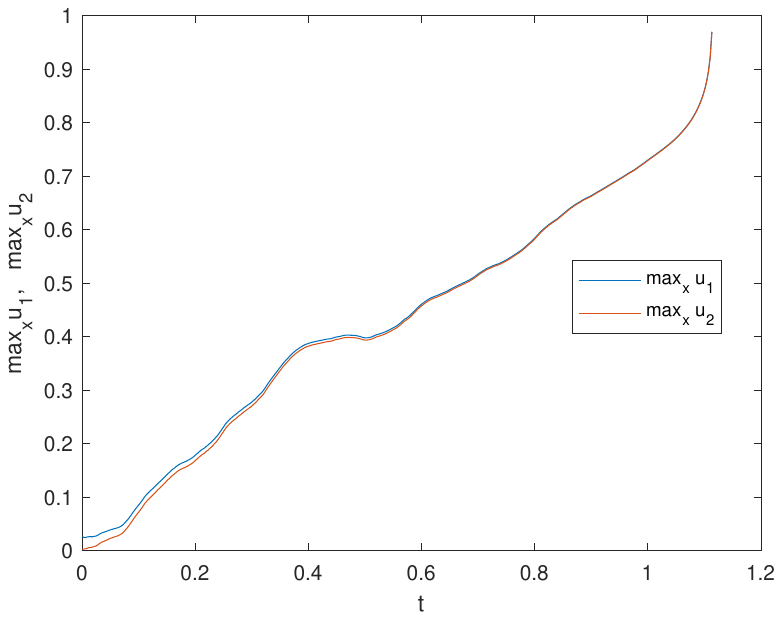}
   \end{minipage}\vspace{-4cm}
   \caption{(a) Realization of the numerical solution of problem~\eqref{b1}--\eqref{b1d} with Robin boundary conditions, using parameters $\beta = \beta_c = 1$, $H = 0.6$, $\lambda_{ij} = 0.08$, $k_{i1} = 0.008$, $k_{i2} = 0.008$ for $i, j = 1,2$, spatial discretization size $M = 102$, time steps $N = 10^4$, and time step size $r = 0.1$. The initial condition is given by $u_i(x,0) = c\,x(1 - x)$ with $c = 0.1$ for $i = 1,2$.
(b) Plot of $\|u_1(\cdot,t)\|_\infty$ and $\|u_2(\cdot,t)\|_\infty$ over time, corresponding to a different realization with the same parameter values as in panel (a).
}\label{Fig_sim1}
\end{figure} 	
	
Next, we present a simulation of the problem under the same setting, but with homogeneous Dirichlet boundary conditions. The parameters used in Figure~\ref{Fig_sim2} are chosen as follows: $H = 0.6$, $\lambda_{ij} = 1$, $k_{i1} = 0.01$, $k_{i2} = 0.001$ for $i, j = 1,2$, and the initial condition is given by $u_i(x,0) = c\,x(1 - x)$ with $c = 0.1$, for $i = 1,2$. In panel (a), the solution is shown as a function of space and time, while in panel (b), the supremum norms $\|u_1(\cdot,t)\|_\infty$ and $\|u_2(\cdot,t)\|_\infty$ are plotted over time.

A similar quenching behaviour to the case with Robin boundary conditions is observed. However, in this scenario, the quenching appears to be primarily driven by the $u_1$ component, see Figure~\ref{Fig_sim2}.

\begin{figure}[!htb]\vspace{-4cm}\hspace{-1cm}
   \begin{minipage}{0.48\textwidth}
     \centering
     \includegraphics[width=1.2\linewidth]{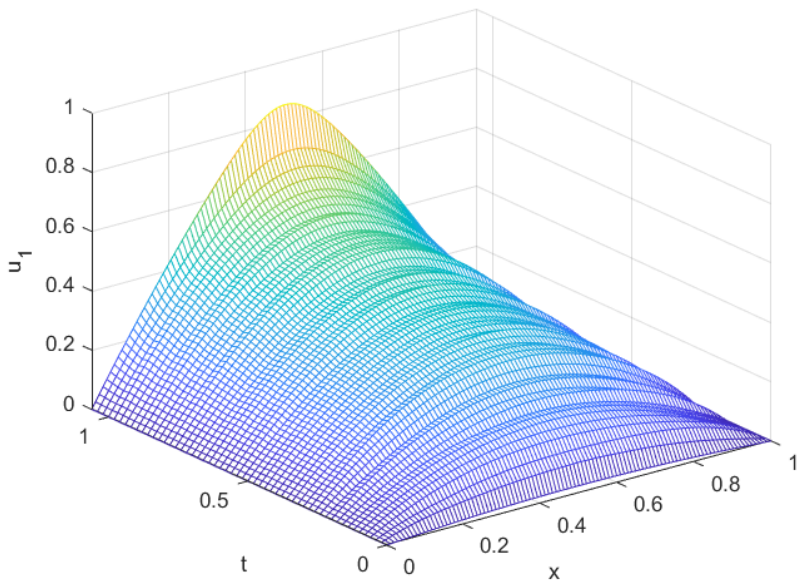}
   \end{minipage}\hfill
   \begin{minipage}{0.48\textwidth}\hspace{-2cm}
     \centering
     \includegraphics[width=1.2\linewidth]{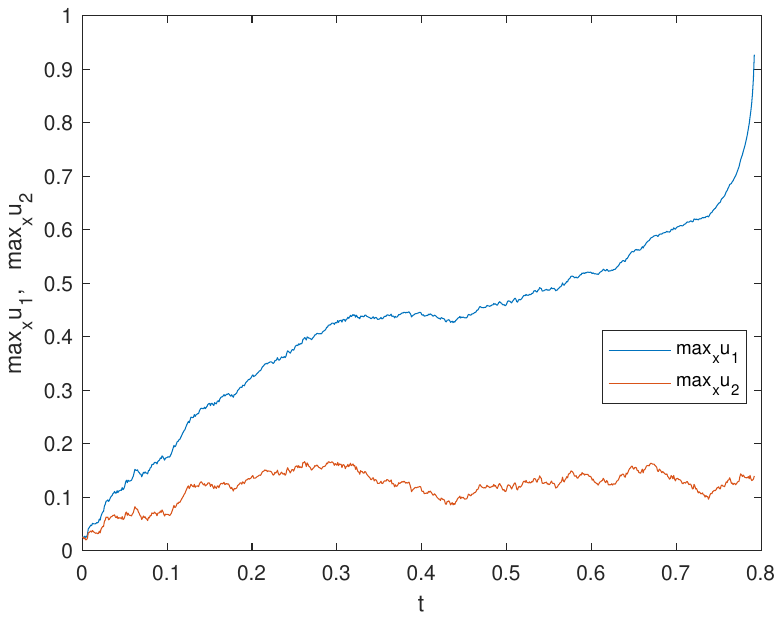}
   \end{minipage}\vspace{-4cm}
   \caption{(a) Realization of the numerical solution of problem~\eqref{b1}--\eqref{b1d} with homogeneous Dirichlet boundary conditions, using parameters $H = 0.6$, $\lambda_{ij} = 1$, $k_{i1} = 0.01$, $k_{i2} = 0.001$ for $i, j = 1,2$, spatial discretization size $M = 102$, number of time steps $N = 10^4$, and time step size $r = 0.1$. The initial condition is given by $u_i(x,0) = c\,x(1 - x)$ with $c = 0.1$ for $i = 1,2$.
(b) Plot of $\|u_1(\cdot,t)\|_\infty$ and $\|u_2(\cdot,t)\|_\infty$ over time, corresponding to a different realization with the same parameter values as in panel (a).
}.\label{Fig_sim2}
\end{figure} 	

\newpage
 

	 \medskip\noindent
     
     \noindent {\bf Data availability:} Data sharing does not apply to this article as no datasets were generated or analyzed during the current study.\\

     \noindent {\bf Disclosure statement:} The authors reported no potential competing interest.


\end{document}